\documentclass{amsart}
\usepackage{amsthm}
\usepackage{amsmath}
\usepackage{amsfonts}
\usepackage[utf8]{inputenc}
\usepackage{makecell}
\usepackage{graphicx}
\usepackage{amsthm} 
\usepackage{amsmath}
\usepackage{amsfonts}
\usepackage{tikz}
\usepackage[capitalize, noabbrev]{cleveref}
\usepackage{tikz-cd}
\usepackage{amssymb}
\usepackage{filecontents}
\usetikzlibrary{3d,calc,intersections}
\usepackage{graphicx}
\usepackage{listings}
\usepackage{pgfplots}
\pgfplotsset{compat=1.15}
\usepackage{mathrsfs}
\usepackage{tikz-3dplot}
\usepackage{tikz,tikz-3dplot} 
\usetikzlibrary{arrows}
\definecolor{Frangipane}{RGB}{255, 217, 179}
\definecolor{Peach}{RGB}{255, 204, 152}

\usepackage{subcaption}

\newtheorem{theorem}{Theorem}[section]

\newtheorem{cor}[theorem]{Corollary}
\newtheorem{prop}[theorem]{Proposition}
\newtheorem{lemma}[theorem]{Lemma}
\theoremstyle{remark}{}
\newtheorem{rmk}[theorem]{Remark}
\theoremstyle{definition}

\newtheorem{defn}[theorem]{Definition}

\newcommand{\C}{\mathbb C}
\newcommand{\R}{\mathbb R}

\newcommand{\SL}[1]{\mathrm{SL}_2 (#1)}
\newcommand{\Mod }{\mathrm{Mod}}

\newcommand{\Z}{\mathbb Z}

\title{Echoes in genus three of Teichm\"uller curves in genus two}
\author{Thomas Le Fils}
\date{}
\begin{document}

\begin{abstract}
We classify the Teichm\"uller curves in the moduli space of genus three Riemann surfaces $\mathcal M_3$ that are obtained by a covering construction from a primitive Teichm\"uller curve in  $\mathcal M_2$.
We describe the action on homology modulo two of the affine groups of translation surfaces generating these primitive curves.
We also classify the $\SL \Z$-orbits of square-tiled surfaces in $\mathcal H(2, 2)$ that are a cover of a genus two one.
\end{abstract}
\maketitle

\section{Introduction}
\subsection{Teichm\"uller curves}
The moduli space of algebraic curves $\mathcal M_g$ is a central object in the geometry of surfaces that parametrises closed Riemann surfaces of genus $g$.
It is endowed with a complex structure and a complete metric, the \emph{Teichm\"uller metric}, that corresponds to its Kobayashi metric.
The geodesics for this metric naturally come in families called \emph{complex geodesics}: for each geodesic $\gamma\colon \R\to \mathcal M_g$, there exists a totally geodesic immersion $f\colon \mathbb H^2\to \mathcal M_g$ such that $\gamma= f_{|i\mathbb R_{>0}}$.
Some exceptional complex geodesics have image of finite area: the \emph{Teichm\"uller curves}\footnote{We only consider Teichm\"uller curves generated by squares of abelian differentials, see \cref{rappels}.}.
A central problem in the theory of Teichm\"uller and flat surfaces is to classify these rare objects, see for example the surveys \cite{M18, MCM23}.

Given a Teichm\"uller curve $V\subset \mathcal M_g$, one can obtain infinitely many others in moduli spaces of curves of higher genera by covering constructions: its \emph{echoes}. 
A Teichm\"uller curve that is not an echo is called \emph{primitive}.
In a series of groundbreaking articles \cite{MCM03, MCM05, MCM06, MCM07}, McMullen gave a complete classification of the primitive Teichm\"uller curves in $\mathcal M_2$. 
There exists an infinite series of Teichm\"uller curves $W_D\subset \mathcal M_2$ called the \emph{Weierstrass curves}. They are indexed by their discriminant, an integer $D\geqslant 5$ such that $D\equiv 0, 1\mod 4$, and their spin $\epsilon \in \{0, 1\}$ when $D > 9$ satisfies $D\equiv 1\mod 8$.
We study the echoes of each Weierstrass curve in $\mathcal M_3$.

\begin{theorem}\label{thWD}
Let $V\subset \mathcal M_2$ be a Weierstrass curve of discriminant $D$. If $D\equiv 5\mod 8$ then $V$ has three echoes in $\mathcal M_3$, and  five otherwise.
\end{theorem}
McMullen showed in \cite{MCM06} that the Weierstrass curves with non-square discriminant account for all primitive Teichm\"uller curves in $\mathcal M_2$, except for one curve, the decagon curve $\mathcal D$.

\begin{theorem}\label{thDecagone}
The curve $\mathcal D$ has three echoes in $\mathcal M_3$.
\end{theorem}
Teichm\"uller curves in genus three can be divided into three categories. The first consists of the primitive ones, for which there is a conjectural classification, see \cite[Problem 5.4]{MCM23}. 
The second category is formed by the echoes of Teichm\"uller curves in genus two, and the third consists of the echoes of the only Teichm\"uller curve in genus one. The present article is a contribution to the classification of Teichm\"uller curves in genus three, studying this second category.
\subsection{The flat viewpoint}
A translation surface is a surface obtained by gluing sides of polygons together with translations.
Equivalently, it is the datum of a pair $(X, \omega)$ where $X\in \mathcal M_g$ and $\omega$ is a holomorphic one-form on $X$, see \cref{rappels}.
The group $\SL \R$ acts on the moduli space $\Omega\mathcal M_g$ of translation surfaces of genus $g$, by acting on polygons.
The Veech group of $(X, \omega)$ is its stabiliser for this action, and is denoted $\mathrm{SL}(X, \omega)$. When $\mathrm{SL}(X, \omega)$ is a lattice, $(X, \omega)$ is said to be a Veech surface.
The Teichm\"uller curves are obtained by projecting the orbits of Veech surfaces to $\mathcal M_g$. This produces Teichm\"uller curves, since $\mathbb H^2=\mathrm{SO}_2(\mathbb R)\backslash \mathrm{SL}_2(\mathbb R)$:
\[
\begin{tikzcd}[column sep=huge]
\mathrm{SL}_2(\mathbb R) \arrow{r} \arrow{d} \arrow{d}
  & \Omega\mathcal M_g \arrow{d} \\
 \mathbb H^2/\mathrm{SL}(X, \omega) \arrow{r}  & \mathcal M_g.
\end{tikzcd}
\]
Therefore our results can be expressed as a classification of the $\SL \R$-orbits of genus three Veech surfaces covering a genus two translation surface. For example we can rephrase \cref{thDecagone} as follows.
\begin{cor}\label{corRelDec}
The genus three Veech surfaces covering a genus two surface with two singularities, and not covering a genus one surface, form three $\mathrm{GL}^+_2(\mathbb R)$-orbits.
\end{cor}
\subsection{Square-tiled surfaces}
Our analysis includes the non-primitive Teichm\"uller curves in $\mathcal M_2$ that are generated by a form in $\mathcal H(2)$: the Weierstrass curves with square discriminant.
They are generated by square-tiled surfaces, that is translation surfaces obtained by gluing unit squares along their sides.
Hubert-Leli\`evre \cite{HL06} and McMullen \cite{MCM05} gave a complete classification of their $\SL \R$-orbits. 
More precisely, they gave a classification of the square-tiled surfaces belonging to this orbit, or in other words, of their $\SL \Z$-orbit. We analyse the lifts of these translation surfaces in genus three.
We say that a square-tiled surface is reduced if the set of its relative periods is $\mathbb Z+i\mathbb Z$.
\begin{theorem}\label{thSTS}
There are ten $\mathrm{SL}_2(\mathbb Z)$-orbits of reduced square-tiled surfaces with $2n$ tiles in $\mathcal H(2, 2)$, that cover a genus two one for all odd $n\geqslant 5$.
There are five orbits if $n\geqslant 4$ is even or $n=3$.
\end{theorem}
Similarly to \cite[Corollary 1.5]{MCM05}, we can rephrase this statement purely topologically. Two branched covers $p_1, p_2\colon S\to T$ are of the same type if there exist orientation-preserving homeomorphisms $s, t$ such that the following diagram commutes:
\[ \begin{tikzcd}
S \arrow{r}{s} \arrow[swap]{d}{p_1} & S \arrow{d}{p_2} \\%
T \arrow{r}{t}& T.
\end{tikzcd}
\]

A square-tiled surface gives rise to a branched cover of the torus $T$.
Only some orbits of \cref{thSTS} give a surjection $\pi_1(S)\to \pi_1(T)$.
We characterise the ones that do and classify the types of the corresponding branched covers.
\begin{theorem}\label{thTypes} Let $S$ be a genus three closed surface and $T$ a torus.
For every $d\geqslant 4$ even, or $d=3$, there are exactly three types of degree $2d$ branched covers $f\colon S\to \mathrm T$ branched over just two points and surjective on homology, that are invariant by a fixed-point free involution. For $d\geqslant 5$ odd, there are exactly seven types of these.
\end{theorem}

\subsection{Monodromy representation}
Let us consider the finite extension of the Veech group consisting of elements of the mapping class group appearing in the identifications $A\cdot\omega\simeq \omega$: the affine group $\mathrm{Aff}^+(X, \omega)$. 
Since the mapping class group acts on homology, we get a \emph{monodromy representation}
\[
\rho\colon \mathrm{Aff}^+(X, \omega)\to \mathrm{Sp}\left (H_1(X, \mathbb Z)\right ).
\]
We will observe that classifying the echoes of Teichm\"uller curves amounts to gaining understanding of the \emph{monodromy groups} of their generators: the images of $\rho$.
An important tool to analyse these groups comes from the fact that they must commute with real multiplication.
Let $\mathcal O_D = \mathbb Z[x]/ (x^2 + ax + b)$, where $D = a^2 - 4b$.
The Jacobian $\mathrm{Jac}(X)$ of $X\in \mathcal M_2$ admits real multiplication by $\mathcal O_D$ if there is an embedding $\mathcal O_D\subset \mathrm{End}(\mathrm{Jac}(X))$ in the self-adjoint endomorphisms, generating a proper subring of $\mathrm{End}(\mathrm{Jac}(X))$.
If $(X, \omega)$ generates a Weierstrass curve of discriminant $D$, then $\mathrm{Jac}(X)$ admits real multiplication by $\mathcal O_D$, such that $\mathcal O_D\cdot \omega\subset \mathbb C\omega$. See the reminders of \cref{rappels} for more details and references.
\begin{prop}\label{commutationReal}
Every element in the monodromy group of $(X, \omega)$ generating a Weierstrass curve commutes with real multiplication.
\end{prop}
This observation is proven in \cite[Lemma 6.2]{FL23} when $D$ is not a square. Commutation with real multiplication also appears in the work of Mukamel \cite{M14}.
It will allow us to explicitly describe the \emph{monodromy groups modulo two}: the images of the monodromy groups in $\mathrm{Sp}\left (H_1(X, \mathbb Z/2\mathbb Z)\right )$, recovering a theorem of \cite{GP23}.
\begin{theorem}\label{thdihedral}
The monodromy groups modulo two of Veech surfaces in  $\mathcal H(2)$, and of the regular decagon, are explicit finite dihedral groups.
\end{theorem}
The action of affine groups on homology has received a lot of attention recently. 
McMullen studied them modulo two for regular polygons in \cite{MCM23_1, MCM23_2}.
Guitérrez-Romo and Pardo described the action of affine groups of Veech surfaces in $\mathcal H(2)$ on Weierstrass points in \cite{GP23}. We recover their result with new methods, and provide a uniform treatment for both the arithmetic and non-arithmetic case.
Freedman and Lucas in \cite{FL23} described in particular the monodromy groups modulo $p$ of primitive Teichm\"uller curves in genus two, for an infinite set of primes $p \geqslant 3$, that can be computed but depend on $D$.

\subsection{Pure echoes}
More generally, one can study Teichm\"uller curves that are generated by a translation surface covering another {via} an unramified cover. 
We call these Teichm\"uller cuves the \emph{pure echoes} of the original Teichm\"uller curve $V$.
We show as a consequence of our study that their numbers can be made arbitrarily large.
We also observe that it follows from a theorem of Funar and Lochak \cite{FL18} that for any primitive Teichm\"uller curve $V\to \mathcal M_g$, there exist infinitely many $g'$ and Teichm\"uller curves $V\to \mathcal M_{g'}$.
\begin{theorem}\label{thHigher}
There can be arbitrarily large numbers of pure echoes of a Teichm\"uller curve in $\mathcal M_g$.
Any primitive Teichm\"uller curve has infinitely many copies in higher genera moduli spaces.
\end{theorem}

\subsection{Organisation of the article}
In \cref{rappels}, we recall some known facts about translation surfaces, abelian differentials, and the classification of Teichm\"uller curves in genus two. 
\cref{lifts} defines echoes, and relates them with actions of affine groups on sets of covers.
We give explicit generators for the monodromy groups of abelian differentials generating the Weierstrass curves in \cref{explicit}.
In \cref{sectionDec}, we compute the monodromy group of the regular decagon, prove \cref{thDecagone} and deduce \cref{corRelDec}.
Then in \cref{sectionReal}, we establish \cref{commutationReal}.
In \cref{sectionWD} we explicitly give the monodromy groups modulo two of generators of the Weierstrass curves and prove \cref{thdihedral}. We deduce \cref{thWD} and \cref{thSTS} from this result.
We then determine the covers of reduced square-tiled surfaces in $\mathcal H(2)$ that give rise to a surjection $\pi_1(S)\to \pi_1(T)$ and show \cref{thTypes}.
Finally in \cref{sectionHigher} we consider echoes in moduli spaces of higher genera and prove \cref{thHigher}.

\subsubsection{Acknowledgements}
I wish to thank Erwan Lanneau for his encouragement and interest.
\section{Abelian differentials}\label{rappels}
In this section we recall some well-known facts about abelian differentials, translation surfaces, Teichm\"uller curves, and their classification in genus two.
This presentation is classical and inspired by the survey \cite{MCM23}. 
\subsection{Teichm\" uller spaces}
\subsubsection{The classical Teichm\"uller space.}
Let us fix $S$ a closed topological surface of genus $g\geqslant 2$. Let $\mathrm{Homeo}^+(S)$ be the group of orientation-preserving homeomorphisms of $S$, and $\mathrm{Homeo}_0(S)$ be its subgroup of elements isotopic to the identity.
Recall that the Teichm\" uller space $\mathcal T(S)$ of $S$ is the space of complex structures on $S$, modulo the natural action of $\mathrm{Homeo}_0(S)$, see \emph{e.g.} \cite[Section 11.1.1]{FarbMargalit}. 
The Teichm\"uller space $\mathcal T(S)$ carries a natural topology for which it is a manifold homeomorphic to $\R^{6g-6}$, see \cite[Theorem 10.6]{FarbMargalit}.
\subsubsection{The Teichm\"uller space of abelian differentials}
We now define the space parametrising the marked translation surfaces of genus $g$, that will be naturally a vector bundle over $\mathcal T(S)$, with the zero section removed.
A translation surface gives a complex structure on the underlying topological surface since the polygons are embedded in the plane. It also gives a non-zero holomorphic one-form on it: the form $dz$ is invariant by translations. The datum of a pair of a closed Riemann surface and a holomorphic one-form on it is actually equivalent to the datum of a translation surface, see \textit{e.g.} \cite[Section 3.3]{Z06}.
\begin{defn}
The Teichm\" uller space of abelian differentials on $S$, denoted by $\Omega\mathcal T(S)$, is the space of $(X, \omega)$ where $X$ is a complex structure on $S$ and $\omega$ is a non-zero holomorphic one-form on $X$, modulo the action of $\mathrm{Homeo}_0(S)$.
\end{defn}
We will sometimes denote by $\Omega\mathcal T_g$ the space $\Omega\mathcal T(S)$ where $S$ is an implicit fixed topological surface of genus $g$.
\subsubsection{Linear charts}
Integrating the form $\omega$ gives local \emph{linear charts} in $\mathbb R^2$ whose transition functions are translations. One can also recover $\omega$ from the charts.
From this point of view, a translation surface can be seen as a $(G, X)$-structure in the sense of Thurston, see \cite[Chapter 3]{T97}.
\subsection{Mapping class group and moduli space}
We denote by $\Mod(S)$ the mapping class group of $S$: $\Mod(S) = \mathrm{Homeo}^+(S)/\mathrm{Homeo}_0(S)$. Recall that the moduli space of curves $\mathcal M_g$ is defined to be $\mathcal T_g/\Mod(S)$.
Similarly  we define the moduli space of abelian differentials $\Omega\mathcal M(S)$ to be the set of couples $(X, \omega)$ modulo the action of the full group of homeomorphisms $\mathrm{Homeo}(S)$, or in other words,
\[\Omega \mathcal M(S) = \Omega\mathcal T(S) / \Mod(S).\]
As before, we denote by $\Omega\mathcal M_g$ the space $\Omega\mathcal M(S)$ where $S$ is a surface of genus $g$.
When $(X_0,\omega_0)$ and $(X_1, \omega_1)$ are two elements of $\Omega\mathcal T(S)$, we write $\omega_0 \simeq \omega_1$ to indicate that they define the same element of the moduli space.
\subsection{Strata of abelian differentials}
The spaces $\Omega\mathcal T_g$ and $\Omega\mathcal M_g$ are both naturally stratified. Indeed an abelian differential $\omega$ on a Riemann surface of genus $g$ has $2g-2$ zeroes counted with multiplicity. Therefore these spaces are stratified by the subspaces where the number of zeroes and their multiplicity are prescribed.
The strata are denoted by $\Omega\mathcal T(n_1, \ldots, n_k)$ and, respectively, $\Omega\mathcal M(n_1, \ldots, n_k)$, where $\sum_i n_i = 2g-2$.
Another common notation for the strata of the moduli space of abelian differentials is the following: $\mathcal H(n_1, \ldots, n_k) = \Omega\mathcal M(n_1, \ldots, n_k)$.
\subsection{The $\mathrm{SL}_2(\mathbb R)$-action}
The moduli space of abelian differentials admits an action of $\SL \R$. Indeed, matrices act on polygons and this action preserves parallelity. 
The projection of an orbit under the action of $\SL \R$ to the moduli space of curves defines a \emph{complex geodesic} for the Teichm\" uller metric: a totally geodesic immersion of $\mathbb H^2$, see for example \cite[Proposition 11.18]{FarbMargalit}.
\[\begin{tikzcd}[column sep=huge]
\mathrm{SL}_2(\mathbb R) \arrow{r} \arrow{d} \arrow{d}
  & \Omega\mathcal M_g \arrow{d} \\
 \mathbb H^2 = \mathrm{SO}_2 \backslash \mathrm{SL}_2(\R)\arrow{r}  & \mathcal M_g
\end{tikzcd}\]
The moduli space of quadratic differentials parametrises half-translation surfaces: surfaces obtained by gluing sides of polygons with translations, and $z\mapsto -z$.
This space carries a similar $\mathrm{SL}_2(\mathbb R)$-action, and the complex geodesics in $\mathcal M_g$ are all obtained by the discussion above, see \cite[Proposition 11.18]{FarbMargalit}.
In the present article, we only consider complex geodesics generated by abelian differentials, or their squares in the space of quadratic differentials.
\subsection{Affine and Veech groups}
The stabiliser of $(X,\omega)\in \Omega\mathcal M_g$ under the action of $\SL \R$ is called its \emph{Veech group} and is denoted by $\mathrm{SL}(X, \omega)$:$$\mathrm{SL}(X, \omega) = \{A\in \mathrm{SL}_2(\R)\mid A\cdot \omega \simeq \omega\}.$$
We define $\mathrm{Aff}^+(X, \omega)$ the \emph{affine group} of $(X, \omega)$ to be the group of $f\in \mathrm{Homeo}^+(S)$ such that $f$ preserves the zeroes of $\omega$ and is affine in linear charts.
We have a short exact sequence, where $\mathrm{Tr}(X, \omega)$ is the group of homeomorphisms of $S$ that are translations in linear charts.
\[1\to \mathrm{Tr}(X, \omega)\to \mathrm{Aff}^+(X, \omega)\to\mathrm{SL}(X, \omega)\to 1.\]
We say that $(X, \omega)$ is a {Veech} surface when its Veech group is a lattice in $\SL \R$.
We denote by $d$ the map taking $f\in \mathrm{Aff}^+(X, \omega)$ to its {derivative}, $d(f)\in \mathrm{SL}(X, \omega)$.
\subsection{Periods}
Given $(X, \omega)\in \Omega\mathcal T(S)$, we define its period map 
$\mathrm{Per}\colon H_1(S, \mathbb Z)\to \C$ to be the homomorphism
\[
\gamma\mapsto \int_\gamma \omega.
\]
We will sometimes refer to $\mathrm{Per}(H_1(S, \mathbb Z))$ as the set of absolute periods, by opposition to the relative periods defined similarly $\mathrm{RelPer}\colon H_1(S, Z, \mathbb Z)\to \C$, where $Z\subset S$ is the set of zeroes of $\omega$. 
The relative periods give local homeomorphisms $U\subset \Omega\mathcal T(n_1, \ldots, n_k)\to \C^{2g+k-1}$, and endow the strata with a manifold structure.

\subsection{Dehn twist and cylinder decompositions}\label{DehnTwistAnd}
A saddle connection on a translation surface is a straight line joining two conical points (\textit{i.e.} two zeroes of the one-form), that has no conical point in its interior.
A Veech surface admits a cylinder decomposition in the direction of each of its saddle connections. 
In other words, for any $s\in \mathbb S^1$ such that there is a saddle connection with slope $s$, the straight lines in the direction of $s$ decompose the surface into cylinders $C_i$.
The modulus of $C_i$ is $m_i = h_i/l_i$, where $l_i$ is its length and $h_i$ its height.
If the cylinder $C_i$ is horizontal, then it is preserved by the right-handed Dehn twist: 
\[
\begin{pmatrix}
    1 & 1/m_i\\
    0 & 1\\
\end{pmatrix}.
\]
The moduli $m_i$ are commensurable.
Let $m$ be the lowest common multiple of the $1/m_i$. 
The affine group of the surface contains the element whose restriction to each $C_i$ is the right-handed Dehn twist to the power $m\cdot m_i$.
\subsection{Monodromy groups}
There is a natural map $\mathrm{Aff}^+(X, \omega)\to \Mod(S)$.
The action on homology of $\Mod(S)$ respects the symplectic form $(\gamma, \delta)\mapsto \gamma\cdot \delta$ given by the algebraic intersection.
This gives a surjection $\Mod(S)\to \mathrm{Sp}\left (H_1(S, \mathbb Z)\right )$, see for example \cite[Theorem 6.4]{FarbMargalit}.
By composing these maps one obtains an action on homology, and the monodromy representation:
\[\mathrm{Aff}^+(X, \omega)\to \mathrm{Sp}\left (H_1(S, \mathbb Z)\right ).\]
We are mostly interested in the action on homology modulo two, hence we will consider the projection $\rho\colon \mathrm{Aff}^+(X, \omega)\to \mathrm{Sp}\left (H_1(S, \mathbb Z/2\mathbb Z)\right )$.
Let us call the \emph{monodromy group modulo two} of $(X, \omega)$ the image of $\rho$.
\subsection{Teichm\" uller discs and Teichm\"uller curve}
Given $(X, \omega)\in \Omega\mathcal M_g$, the \emph{Teichm\"uller disc} generated by $\omega$ is $\SL \R$-orbit of $\omega$ in $\Omega \mathcal M_g$.
A \emph{Teichm\"uller curve} is the projection of a Teichm\"uller disc in $\mathcal M_g$ that has finite area. This happens exactly when the forms $(X, \omega)$ generating the Teichm\"uller disc are Veech surfaces.

\subsection{The classification of McMullen in genus two}

In a series of groundbreaking articles \cite{MCM03, MCM05, MCM06, MCM07}, McMullen gave a complete classification of the primitive Teichm\" uller curves in genus two.
Let us recall the definition of the Weierstrass curves. For any $D\geqslant 5$, with $D\equiv 0, 1 \mod 4$, the Weierstrass curve $W_D$ is the locus of Riemann surfaces $X\in \mathcal M_2$ such that 
\begin{enumerate}
    \item The Jacobian of $X$ admits real multiplication by $\mathcal O_D$,
    \item There exists an abelian differential $\omega$ on $X$ with a single double zero, such that $\mathcal O_D\cdot\omega\subset \mathbb C\omega$.
\end{enumerate}
Recall that the Jacobian of $X$ is the complex torus $\mathrm{Jac}(X) = \Omega^\vee(X)/H_1(S, \mathbb Z)$.
Here $\Omega^\vee(X)$ is the dual of $\Omega(X)$, the vector space of holomorphic one-forms on $X$, and $H_1(S, \mathbb Z)$ is a lattice in $\Omega^\vee(X)$ by $\gamma\mapsto \int_\gamma \cdot$.
An endomorphism of $\mathrm{Jac}(X)$ comes from a unique $\C$-linear map $T\colon \Omega^\vee(X)\to \Omega^\vee(X)$ preserving the lattice $H_1(S, \mathbb Z)$.
We say that $T$ is self-adjoint if for any $\gamma, \delta\in H_1(S, \mathbb Z)$, we have $T(\gamma)\cdot \delta = \gamma \cdot T(\delta)$.
The dual of $T$ acts  on $\Omega(X)$, giving a meaning to $\mathcal O_D\cdot \omega \subset \mathbb C\omega$.
The subring $\mathcal O_D\subset \mathrm{End}(\mathrm{Jac}(X))$ is said to be proper when for $T\in \mathrm{End}(\mathrm{Jac}(X))$, if $T$ is not in $\mathcal O_D$, then neither is $nT$ for all $n\in \mathbb Z\setminus \{0\}$.
For more information on real multiplication in this context, we refer to \cite{MCM07, MCM23}.
McMullen showed that the components of $W_D$ are Teichm\"uller curves, that are primitive if and only if $D$ is not a square.
He also proved \cite{MCM05} that they are connected unless $D\equiv 1\mod 8$, $D > 9$, in which case $W_D$ has two connected components, distinguished by a spin inviariant $\epsilon\in \{0, 1\}$.
Veech proved in the seminal article \cite{V89} that the regular polygons with even number of sides generate Teichm\"uller curves. 
In particular, the regular decagon defines a Teichm\"uller disc in $\mathcal H(1, 1)$ and gives a Teichm\"uller curve in $\mathcal M_2$.
\begin{theorem}[McMullen \cite{MCM06}]
The curve generated by the decagon and the Weierstrass curves with non-square discriminant account for all the primitive Teichm\"uller curves in $\mathcal M_2$. 
\end{theorem}
\subsection{Square-tiled surfaces}
Square-tiled surfaces, or origamis, form an important class of Veech surfaces of combinatorial nature.
Intuitively, a square-tiled surface is a translation surface obtained by gluing sides of unit squares.
A surface tiled by $n$ squares is entirely characterised by the following two permutations: one taking a tile to its adjacent tile on its right, and the other taking it to its adjacent tile on top of it.
Square-tiled surfaces can thus be thought of as pairs of permutations $(\alpha, \beta)\in \mathfrak S_n$, such that $\langle \alpha, \beta \rangle$ acts transitively on $\{1, \ldots, n\}$, up to simultaneous conjugacy, see for example \cite[Definition 6 and Remark 9]{M22}.
Another way of defining them is to consider covers of tori.
Let us denote by $T$ the standard torus: the translation surface $\mathbb C/ {(\mathbb Z+ i\mathbb Z)}$ equipped with the form  $\omega = dz$.
\begin{defn}
A square-tiled surface is a translation surface obtained by pulling back the standard torus  {via} a cover $S\to T$ that sends its branch points to a single point.
\end{defn}

We say that a square-tiled surface is \emph{reduced} when the set of its relative periods is $\mathbb Z+ i\mathbb Z$.
In \cite{HL06}, Hubert and Leli\`evre classified the $\mathrm{SL}_2(\mathbb Z)$-orbits of reduced origamis in $\mathcal H(2)$ tiled with a prime number of tiles. McMullen then classified them for all numbers of tiles in \cite{MCM05}. 
\begin{theorem}(Hubert-Leli\`evre, McMullen)
There are two $\mathrm{SL}_2(\mathbb Z)$-orbits of reduced square-tiled surfaces in $\mathcal H(2)$ with $n$ squares if $n\geqslant 5$ is odd. There is only one orbit for $n\geqslant 4$ even or $n=3$.
\end{theorem}
There is still no similar classification for square-tiled surfaces in $\mathcal H(1, 1)$, but there are results in this direction in \cite{D19}.
\section{Lifts of abelian differentials}\label{lifts}
\subsection{Covering maps}
Throughout this article, unless otherwise stated, we will consider covering maps up to upper equivalence.
That is, we consider that the covers $p_i\colon S\to \Sigma$ are equal for $0\leqslant i\leqslant 1$ when there exists a homeomorphism $F\colon S\to S$ such that the following diagram commutes
\[
\begin{tikzcd}
    S \arrow{rr}{F} \arrow[swap]{dr}{p_0} & & S \arrow{dl}{p_1} \\[6pt]
    & \ \Sigma.
\end{tikzcd}
\]
The mapping class group $\Mod (\Sigma)$ acts on the set of unramified covers $S\to \Sigma$ by pushforward.
Namely if $p\colon S\to \Sigma$ and $f\in \Mod(\Sigma)$, then $f\cdot p = f\circ p$.
\subsection{Definitions}
Let $f\colon S\to \Sigma$ be a possibly branched covering map. The lift, or pullback, of $(X, \omega)\in \Omega\mathcal T(\Sigma)$ by $f$  is defined to be the abelian differential in $\Omega\mathcal M(S)$ whose linear charts are given by the composition of $f$ and charts of $(X, \omega)$. We denote by $f^*(X, \omega)$, or $f^*(\omega)$, the lift of $(X, \omega)$ by $f$. We can now define the echoes.
\begin{defn}
An \emph{echo} of a Teichm\"uller curve generated by $(X, \omega)$ is a Teichmüller curve generated by a lift of $(X, \omega)$.
\end{defn}
It is easy to verify that this definition is consistent: being an echo of $V$ does not depend on the choice of the generator of $V$.
In the present article, we are mostly interested in unramified covers, and from now on we suppose that $\Phi$ is the set of these covers $S\to\Sigma$.
We define the lifting map: \[\mathcal L\colon \Omega\mathcal T(\Sigma)\times \Phi \to \Omega\mathcal M(S).\]
The lift of a set $X\subset \Omega\mathcal T(\Sigma)$ is defined to be $\mathcal L(X\times \Phi)$,
and the lift of $Y\subset \Omega\mathcal M(\Sigma)$ is the lift of the preimage of $Y$ by $\Omega\mathcal T(\Sigma)\to \Omega
\mathcal M(\Sigma)$.
By definition, we have $\mathcal L(f\cdot \omega, f\cdot p) = \mathcal L(\omega, p)$.
\begin{lemma}\label{invariance}
The map $\mathcal L$ is invariant by the diagonal action of $\Mod(\Sigma)$.
\end{lemma}
Therefore in order to understand all the lifts of a specific $(X, \omega)\in \Omega\mathcal M_g$, it suffices to consider one of its representatives in $\Omega\mathcal T_g$ and all the possible different covers.
\subsection{Genus three surfaces covering a genus two one}
We specialise to the case where $\Sigma$ is a closed surface of genus two, whilst $S$ is of genus three.
We fix these topological surfaces for the rest of the article.
\subsubsection{The set of covers.}
Let us first observe that the covers $S\to \Sigma$ have no ramification point, and that there are fifteen of them. This is also observed for example in \cite[Proposition 2.3]{L18}, where Lanneau and Nguyen use this fact to show that there is only a finite number of Teichm\"uller curves in the Prym locus in $\mathcal H(2, 2)^\mathrm{hyp}$ of fixed discriminant.
\begin{lemma}
Every covering map $S\to \Sigma$ is unramified. The space of these covers is in bijection with ${H}^1(\Sigma, \mathbb Z/2\mathbb Z)\setminus \{0\}$.
\end{lemma}

\begin{proof}
If $f\colon S\to \Sigma$ is a covering map of degree $d$, with total branching order $B$, then the Riemann-Hurwitz formula gives 
$$\chi(S) = d\chi(\Sigma) - B.$$
Since $\chi(S)=-4$ and $\chi(\Sigma)=-2$, we must have $d\leqslant 2$. Moreover if $d=2$ then $B=0$ and $f$ is unramified. Finally, it is impossible to have $d=1$ since otherwise $f$ would be a homeomorphism.
Every degree two cover is normal thus the space of these covers is in bijection with the set of epimorphisms $\pi_1(\Sigma)\to \mathbb Z/2\mathbb Z$, in other words with the non-zero elements of ${H}^1(\Sigma, \mathbb Z/2\mathbb Z)$.
\end{proof}

By duality, we may identify the non-zero elements of ${H}^1(\Sigma, \mathbb Z/2\mathbb Z)$ with the ones of ${H}_1(\Sigma, \mathbb Z/2\mathbb Z)$.
We will in the rest of the text identify the latter with the set of degree two covers $S\to \Sigma$.
The cover corresponding to $\gamma\in {H_1}(\Sigma, \mathbb Z/2\mathbb Z)$ is then obtained by the following construction. 
Consider two copies of $\Sigma$ and a simple closed curve $c\subset \Sigma$ on each copy representing $\gamma$. 
Cutting the two copies along $c$ and gluing them back together produces the cover associated with $\gamma$.
The mapping class group acts transitively on these covers.
\begin{lemma}\label{lemmeMCG}
The mapping class group $\Mod(\Sigma)$ acts transitively on the set of covers $S\to \Sigma$.
\end{lemma}

\begin{proof}
The action of the mapping class group on covers is conjugated to its natural action on ${H}_1(\Sigma, \mathbb Z/2\mathbb Z)\setminus \{0\}$.  
This action is easily seen to be transitive, see for example \cite[Proposition 3.2]{LF23}.
\end{proof}

\subsubsection{Lifts of strata of $\Omega\mathcal M_2$ in $\Omega\mathcal M_3$}
Kontsevich and Zorich gave a complete classification of the connected components of the strata of $\Omega\mathcal M_g$ in \cite{KZ03}. Calderon and Calderon-Salter \cite{C20, CS21} extended these results and studied the connected components of the strata in $\Omega\mathcal T_g$.
We will now identify the connected components of the locus in the strata of $\Omega\mathcal M_3$ of surfaces covering a genus two one.
This was already done in \cite{AN16, AN20}, but we provide a different and more direct proof. 
For this purpose, we will need to know the connected components of $\Omega\mathcal T(1, 1)$ and $\Omega\mathcal T(2)$.
In \cite[Remark 2.7]{C20}, Calderon classifies these connected components, and attributes this result to folklore knowledge.
\begin{lemma}[Folklore]
The space $\Omega\mathcal T(1, 1)$ is connected, while the space $\Omega\mathcal T(2)$ has six connected components, corresponding to the six Weierstrass points.
\end{lemma}
The space $\mathcal H(1,1)$ lifts to a connected set $\widetilde{\mathcal H}(1, 1)$ in $\mathcal H(1, 1, 1, 1)$, as shown in \cite[Proposition 2.17]{AN20}. 
Here we give a similar argument, that illustrates how we can use the action of the mapping class group to study lifts.
\begin{prop}
The set $\widetilde{\mathcal H}(1,1)$ in $\mathcal H(1, 1, 1, 1)$ is connected.
\end{prop}
\begin{proof}
Let us first restrict the lifting map to $\Omega\mathcal T(1, 1)$:
$$\mathcal L\colon \Omega\mathcal  T(1, 1)\times  H_1(\Sigma, \mathbb Z/2)\setminus \{0\}\to \mathcal H(1, 1, 1, 1).$$
Let us show that its image is connected. By \cref{invariance}, $\mathcal L(\omega, \gamma) = \mathcal L(f\cdot\omega, f\cdot\gamma)$ for any $(\omega, \gamma)$ and any $f\in \Mod(\Sigma)$.
Recall that $\Mod(\Sigma)$ acts transitively on ${H_1}(\Sigma, \mathbb Z/2)\setminus \{0\}$. Thus for $\gamma\in  H_1(\Sigma, \mathbb Z/2)\setminus \{0\}$, the image of $\mathcal L$ does not change when restricted to $\Omega\mathcal T(1, 1)\times \{\gamma\}$. The space $\Omega \mathcal T(1, 1)$ is connected hence so is its image $\widetilde{\mathcal H}(1, 1)$.
\end{proof}

It is proven in \cite{KZ03} that $\mathcal H(2, 2)$ has two connected components: the hyperelliptic component $\mathcal H(2)^\mathrm{hyp}$, and the odd component $\mathcal H(2)^\mathrm{odd}$.
To tell which connected component $(X, \omega)$ belongs to, one can proceed as follows.
For each $\gamma\in H_1(\Sigma, \mathbb Z)$ primitive, consider $c\colon \mathbb S^1\to X$ representing $\gamma$, avoiding the zeroes of $(X, \omega)$ with non-zero derivative everywhere.
Taking local linear charts, we can compute $c'\colon \mathbb S^1\to \mathbb C\setminus \{0\}$.
We then define $\mathrm{Ind}(\gamma)$ to be the degree of the map $\mathbb S^1\to \mathbb S^1$ given by $s\to \frac{c'(s)}{|c'(s)|}$.
Given a symplectic basis $(a_1, b_1, a_2, b_2)$ of $X$, we define
\[\mathrm{Arf}(X, \omega) = \sum_i \left (\mathrm{Ind}(a_i) + 1\right ) \left (\mathrm{Ind}(b_i) + 1\right )\mod 2.\]
This does not depend on the choices of the $c$ nor on the choice of the symplectic basis and $(X, \omega)$ is in $\mathcal{H}(2)^\mathrm{odd}$ if and only if $\mathrm{Arf}(X, \omega) = 1$, see also \cite[Section 9.4]{Z06}.
We apply this criterion to exhibit a surface in different connected components of $\widetilde{\mathcal H}(2)$.
Indeed, we easily verify by considering a basis of homology on each of the surfaces of \cref{twoExamples} that they belong to different components of $\mathcal H(2, 2)$.
They are both covering an $L$-shaped surface in $\mathcal H(2)$, that we will study later, see \cref{Lshaped} in \cref{explicit}.

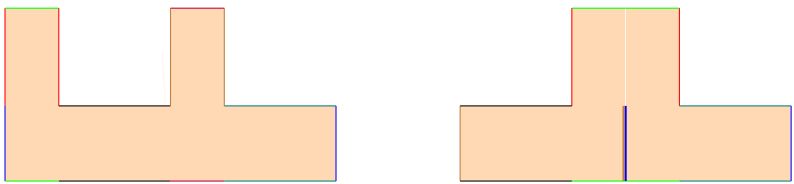
\begin{figure}[h!]
\begin{tikzpicture}[>=stealth,yscale=1,xscale=0.55]  
\coordinate (A) at (2,4);
\coordinate (B) at (2,5.3);
\coordinate (C) at (3.3, 5.3);
\coordinate (D) at (3.3, 4);
\coordinate (E) at (6, 4); 
\coordinate (F) at (6, 3);
\coordinate (G) at (3.3, 3);
\coordinate (H) at (2, 3);
\coordinate (I) at (1.8, 4.7) ;
\coordinate (J) at (6,4);
\coordinate (K) at (6,5.3);
\coordinate (L) at (7.3, 5.3);
\coordinate (M) at (7.3, 4);
\coordinate (N) at (10, 4); 
\coordinate (O) at (10, 3);
\coordinate (P) at (7.3, 3);
\coordinate (Q) at (6, 3);
\coordinate (R) at (5.8, 4.7) ;
\coordinate (AA) at (17,4);
\coordinate (BB) at (17,5.3);
\coordinate (CC) at (18.3, 5.3);
\coordinate (DD) at (18.3, 4);
\coordinate (EE) at (21, 4); 
\coordinate (FF) at (21, 3);
\coordinate (GG) at (18.3, 3);
\coordinate (HH) at (17, 3);
\coordinate (II) at (17,4);
\coordinate (JJ) at (17,5.3);
\coordinate (KK) at (15.7, 5.3);
\coordinate (LL) at (15.7, 4);
\coordinate (MM) at (13, 4); 
\coordinate (NN) at (13, 3);
\coordinate (OO) at (15.7, 3);
\coordinate (PP) at (17, 3);
\coordinate (III) at (16.95, 4);
\coordinate (PPP) at (16.95, 3);
\fill[color=Frangipane] (A) -- (B) -- (C) -- (D) -- (E) -- (J) -- (K) -- (L) -- (M) -- (N) -- (O) -- (P) -- (Q) -- (R) -- (F) -- (G) -- (H) -- cycle ;
\fill[color=Frangipane] (AA) -- (BB) -- (CC) -- (DD) -- (EE) -- (FF) -- (GG) -- (HH) --  (II) -- (JJ) -- (KK) -- (LL) -- (MM) -- (NN) -- (OO) -- (PP) -- cycle;
\draw[color=red] (A) -- (B) ;
\draw[color=green](B) -- (C) ;
\draw[color=red] (C) -- (D) ;
\draw[color=blue] (H) -- (A) ;
\draw (D) -- (E) ;
\draw[color=blue] (N) -- (O);
\draw[color=black] (F) -- (G);
\draw[color=green] (G) -- (H);
\draw[color=brown] (J) -- (K) ;
\draw[color=brown] (L) -- (M) ;
\draw[color=purple] (K) -- (L) ;
\draw[color=purple] (P) -- (Q) ;
\draw[color=teal] (M) -- (N) ;
\draw[color=teal] (O) -- (P) ;
\draw[color=brown] (MM) -- (NN) ;
\draw[color=blue] (EE) -- (FF) ;
\draw[color=green] (CC) -- (KK) ;
\draw[color = green] (GG) -- (OO) ;
\draw[color=red] (CC) -- (DD) ;
\draw[color=red] (KK) -- (LL) ;
\draw[color=teal] (DD) -- (EE) ;
\draw[color=teal] (FF) -- (GG) ;
\draw[color=black] (LL) -- (MM) ;
\draw[color=black] (NN) -- (OO) ;
\draw[thick, color=blue] (AA) -- (HH) ;
\draw[thick, color=brown] (III) -- (PPP) ;
\end{tikzpicture}
\caption{Translation surfaces, in $\widetilde{\mathcal H}(2)^\mathrm{odd}$ on the left, and in $\widetilde{\mathcal H}(2)^\mathrm{hyp}$ on the right.}
\label{twoExamples}
\end{figure}

We now show that the lift $\widetilde{\mathcal H}(2)\subset \mathcal H(2, 2)$ has two connected components.
This result is obtained in \cite{AN16} as a consequence of the classification of affine invariant submanifolds in $\mathcal H(2, 2)$.

\begin{prop}\label{twoCC}
The stratum $\mathcal H(2)$ lifts to two connected components $\widetilde {\mathcal H}^{\mathrm{odd}}(2)$ and $\widetilde {\mathcal H}^{\mathrm{hyp}}(2)$ in $\mathcal H(2, 2)$, one in each connected component of $\mathcal H(2, 2)$.
\end{prop}

Choose a complex structure $X$ on $\Sigma$ and let $\iota\colon \Sigma\to \Sigma$ be a holomorphic involution representing the hyperelliptic involution, the only non-trivial element in the centre of $\Mod(\Sigma)$.
The Weierstrass points of $X$ are the fixed points of $\iota$.
It follows from a remark of Arnold \cite{A68} that a non-zero element in $H_1(\Sigma, \mathbb Z/2\mathbb Z)$ can be identified with the choice of two Weierstrass points on $X$.
\begin{lemma}
The non-zero elements of $H_1(\Sigma, \mathbb Z/2\mathbb Z)$ can be naturally identified with the unordered couples of Weierstrass points of $X$.
\end{lemma}
\begin{proof}
By \cite[Lemma 1]{A68}, the quotient by the hyperelliptic involution gives a well defined injection $H_1(\Sigma, \mathbb Z/2\mathbb Z)\to H_1(S_0^6, \mathbb Z / 2\mathbb Z)$ to the homology of the sphere with six points removed.
The latter homology group is identified with the partitions of the Weierstrass points in two sets by \cite[Lemma 2]{A68}.
We can take a non-separating simple closed curve whose image singles out two Weierstrass points. 
Since $\Mod(\Sigma)$ acts transitively on the non-separating simple closed curves, it is the case of every such curve.
The image of $H_1(\Sigma, \mathbb Z/2\mathbb Z)$ is thus identified with the $15 = {6 \choose 2}$ unordered pairs of Weierstrass points.
\end{proof}
We can now prove \cref{twoCC}.
\begin{proof}
There are at least two of these connected components. 
Indeed we have seen that the surfaces of \cref{twoExamples} do not belong to the same component of $\mathcal H(2, 2)$.
The space $\Omega \mathcal T(2)$ has six connected components, that correspond to the Weierstrass point on which the zero of the form lies.
Since $\Mod(\Sigma)$ acts transitively on $H_1(\Sigma, \mathbb Z/2\mathbb Z)\setminus \{0\}$, the lift of $\Omega \mathcal T(2)$ is covered by the sets 
$\mathcal L(X_i\times \{\gamma\})$ where $X_i$ are the connected components of $\Omega\mathcal T(2)$ for $1\leqslant i\leqslant 6$, and $\gamma\in H_1(\Sigma, \mathbb Z/2\mathbb Z)$ is a fixed non-zero element.
The Dehn twists along a simple closed curve that single out two Weierstrass points interchange them, and leave the others unchanged, see for example \cite[Proposition 3.3]{BB01} or \cite[Section 2.1]{MW19}.
A Dehn twist $T$ along a simple closed curve $c$ representing $\gamma$ will interchange the two Weierstrass points singled out by $\gamma$, and preserves $\gamma$.
Therefore $T$ interchanges the two connected components, say $X_1$ and $X_2$, of $\Omega\mathcal T(2)$ corresponding to the Weierstrass points singled out by $\gamma$. 
Thus they lift to the same sets in $\mathcal H(2, 2)$:
\[\mathcal L(X_1\times \{\gamma\}) = \mathcal L(T(X_1)\times \{T(\gamma)\}) = \mathcal L\left (X_2\times \{ \gamma\}\right ).\]
Similarly, Dehn twists along simple closed curves not intersecting $c$ will permute the other four Weierstrass points, leaving $\gamma$ invariant.
Therefore the lift of $\mathcal H(2)$ has at most two connected components.
\end{proof}
Observe that our proof of \cref{twoCC} gives a criterion to determine the connected component where $(X, \omega)\in \Omega\mathcal T(\Sigma)$ is sent to by $\gamma\in  H_1(\Sigma, \mathbb Z/2\mathbb Z)$.
Indeed it suffices to check whether the zero of $(X, \omega)$ is one of the two Weierstrass points singled out by $\gamma$.
It is the case exactly when the lift is in $\widetilde{\mathcal H}(2)^\mathrm{hyp}$.
Therefore we have the following.
\begin{cor}\label{corHyp}
Among the fifteen lifts of $(X, \omega)\in \Omega\mathcal T(\Sigma)$, exactly five belong to the hyperelliptic component.
\end{cor}
\subsection{Telling the lifts appart}
We now wish to determine when two lifts define the same abelian differential in $\Omega\mathcal M_3$. 
Let us first show that, under an extra assumption, an isomorphism of lifts comes from an isomorphism at the base level.
\begin{lemma}\label{lemmeMartens}
Let $\alpha_i\in \Omega\mathcal T(\Sigma)$, and $p_i\colon S\to \Sigma$ be covers, for $0\leqslant i\leqslant 1$.
Consider $\omega_0= p_0^*(\alpha_0)$ and $\omega_1 = p_1^*(\alpha_1)$. 
Suppose that the forms $\omega_i$ have only one translation with quotient of genus two.
Every homeomorphism $F\colon S\to S$ such that $F\cdot \omega_0 = \omega_1$ is induced by a homeomorphism $f\colon \Sigma\to \Sigma$ such that $f\circ p_0 = p_1\circ F$ and $f\cdot \alpha_0 = \alpha_1$.
\end{lemma}
\begin{proof}
We wish to check that $F$ passes through the quotient by the deck transformations to a map $f\colon \Sigma\to \Sigma$.
The deck transformations of $p_i$ are translations with quotient of genus two. Since $\omega_i$ have only one such translation, the deck transformations are conjugated by $F$.
Therefore $F$ passes to a map $f$. 
\end{proof}
\begin{rmk}
\cref{lemmeMartens} does not hold without this extra assumption.
Indeed, there exist non-isomorphic translation surfaces in genus two with isomorphic double covers.
For example, the escalator square-tiled surface defined by
$\alpha = (1,2)(3,4)(5,6)(7,8)$ and $\beta = (2,3)(4,5)(6,7)(8,1)$, see \cite[Figure 4]{PW17}.
It is a cover of both the origamis defined by $\alpha = (1,2)(3,4)$ and $\beta = (1,3)$ and $\alpha = (2,3)$, $\beta=(1,2)(3,4)$ that are not isomorphic.
\end{rmk}
We can now describe, under our assumption, when two lifts define the same element in $\Omega\mathcal M_3$.
This happens exactly when the pairs $(\omega, \gamma)$ differ by the diagonal action of the mapping class group on $\Omega\mathcal T(\Sigma)\times H_1(\Sigma, \mathbb Z/2\mathbb Z)$, by \cref{lemmeMartens} and the invariance of $\mathcal L$ by the diagonal action of $\Mod(\Sigma)$.

\begin{prop}\label{same}
Let $(X_i, \omega_i)\in \Omega\mathcal T(\Sigma)$ and $p_i\colon S\to \Sigma$ be covering maps, for $0\leqslant i \leqslant 1$.
Suppose that $p^*(\omega_i)$ have only one non-trivial translation.
We have $p_0^*(\omega_0) \simeq p_1^*(\omega_1)$ if and only if there exists $f\in \Mod(\Sigma)$ such that both $f\cdot\omega_0 = \omega_1$ and $ f\cdot p_0  = p_1$.
\end{prop}
\subsection{Lifting Teichmüller discs and Veech groups}
We now derive from \cref{same} a description of the Veech groups of forms covering genus two ones.
We will assume that the lifts only have one non-trivial translation.
Let $(X, \omega)\in \Omega\mathcal T(\Sigma)$.
Let us denote by $[p]$ the orbit of the cover $p$ under the action of $\mathrm{Tr}(X, \omega)$.
Since $\mathrm{Tr}(X, \omega)$ is normal, the group $\mathrm{Aff}^+(X, \omega)$ acts on these orbits. 
We thus obtain an action of $\mathrm{SL}(X, \omega)$ on the set of orbits $[p]$.
\begin{cor}\label{veech_cover}
Let $(X_0, \omega_0)\in \Omega\mathcal T(\Sigma)$ and $p\colon S\to\Sigma$ be a covering map. 
Suppose that $(X, \omega) = p^*(X_0, \omega_0)$ has only one translation with quotient of genus two.
The Veech group of $(X, \omega)$ is
\[\mathrm{SL}(X, \omega) = \{A\in \mathrm{SL}(X_0, \omega_0)\mid A\cdot [p] = [p]\}.\]
\end{cor}
\begin{proof}
Let $A\in \mathrm{SL}_2(\mathbb R)$. 
We have $A\cdot p^*(\omega_0) = p^*(A\cdot \omega_0)\simeq p^*(\omega_0)$ if and only if there exists $f\in \Mod(\Sigma)$ such that both $f\cdot p = p$ and $f\cdot \omega_0 = A\cdot \omega_0$ by \cref{same}.
This happens exactly when $A\in \mathrm{SL}(X_0, \omega_0)$ and $A\cdot [p] = [p]$.
\end{proof}
\begin{rmk}\label{rmkInclusion}
If $p\colon S_g\to \Sigma_h$ is a covering map and $\omega\in \Omega \mathcal T(\Sigma_h)$, we always have the inclusion $\{A\in \mathrm{SL}(X, \omega)\mid A\cdot [p] = [p]\}\subset \mathrm{SL}(p^*(X, \omega))$. Indeed a map in $\mathrm{Aff}^+(X, \omega)$ that stabilises $p$ lifts to a map in $\mathrm{Aff}^+(p^*(X, \omega))$ with the same derivative.
\end{rmk}
We also derive from \cref{same} a criterion to decide when two lifts lie in the same Teichmüller disc.
\begin{cor}\label{sameDisc}
Two lifts of $(X, \omega)\in \Omega\mathcal T(\Sigma)$ by the covers $p_i\colon S\to \Sigma$, $0\leqslant i\leqslant 1$, with only one non-trivial translation with quotient of genus two, belong to the same Teichmüller disc if and only if $p_1$ is in the orbit of $p_0$ under the action of $\mathrm{Aff}^+(X, \omega)$.
\end{cor}
\begin{proof}
The forms $p_0^*(\omega)$ and $p_1^*(\omega)$ belong to the same Teichm\"uller disc if and only if there exists $A\in \SL \R$ such that $A\cdot p_1^*(\omega) = p_1^*(A\cdot\omega) \simeq p_0^*(\omega)$.
It follows from \cref{same} that this happens exactly when $p_1$ is the orbit of $p_0$ under the action of $\mathrm{Aff}^+(X, \omega)$.
\end{proof}
We will verify that the assumption of \cref{lemmeMartens} is satisfied in our cases.
\begin{lemma}\label{lemmaNoAuto}
Let $(X, \omega)\in \Omega\mathcal T(\Sigma)$ be a generator of either a Weierstrass curve, or the decagon curve.
Any lift $\tilde {\omega}\in \Omega\mathcal M(S)$ has only one translation with quotient of genus two, unless $(X, \omega)$ generates $W_9$.
\end{lemma}
We postpone the proof of \cref{lemmaNoAuto} to \cref{explicit} for the Weierstrass curves, and to \cref{sectionDec} for the decagon one.
Since a Teichm\"uller curve is generated by a unique Teichm\"uller disc (up to scaling), the echoes of the curve generated by $(X, \omega)$ are classified by the action of the $\mathrm{Aff}^+(X, \omega)$ on homology in these cases.
\begin{cor}\label{corEchoes}
Let $V\subset \mathcal M_2$ be either a Weierstrass curve or the decagon curve. Fix $(X, \omega)\in \Omega\mathcal T(\Sigma)$ generating $V$.
The echoes of $V$ in $\mathcal M_3$ are in correspondence with the orbits of the action of $\mathrm{Aff}^+(X, \omega)$ on $H_1(\Sigma, \mathbb Z/2\mathbb Z)\setminus \{0\}$.
\end{cor}
We prove \cref{corEchoes} for $W_9$ in \cref{explicit}. The other cases follow from \cref{lemmaNoAuto} and \cref{sameDisc}.
\subsection{Echoes are covers}
Under our assumption, each echo $\tilde V$ generated by $(X, \omega)$ of $V$ generated by $(X_0, \omega_0)$ is naturally a cover of $V$ since $\tilde V = \mathbb H^2/\mathrm{SL}_2(X, \omega)$ and $V = \mathbb H^2/\mathrm{SL}_2(X_0, \omega_0)$.
The degree $d = [\mathrm{SL}_2(X, \omega)\colon \mathrm{SL}_2(X_0, \omega_0)]$ of this cover is the size of the orbit of the corresponding $[p]$.
\section{Generators of the monodromy groups of Weierstrass curves}\label{explicit}
In this section we exhibit elements of the monodromy groups of generators of Weierstrass curves.
\subsection{L-shaped billiards}
In \cite{MCM05}, McMullen gave explicit generators for the Teichm\"uller curves $W_D\subset \mathcal M_2$.
These generators are given as surfaces obtained by unfolding a billiard table of the form $L(b, e)$: the L-shaped polygon of \cref{Lshaped}, having top square of side $\lambda = \frac{e + \sqrt{e^2 + 4b}}{2}$, while the remaining rectangle has horizontal sides of length $b$ and vertical side of unit length.
It turns out that unfolding this billiard table leads, after rescaling, to the translation surface that we still denote $L(b, e)$, drawn in \cref{Lshaped}.
The integers $(e,b)$ are chosen so that $b \geqslant 1$, $e\in \{-1, 0, 1\}$ and if $e = 1$, then $b$ is even. We also assume that $e+1 < b$.
We have $\lambda^2 = e\lambda + b$. 
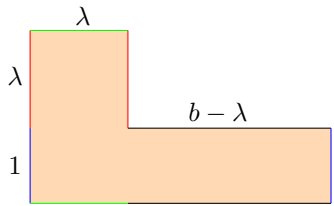
\begin{figure}[h!]
\begin{tikzpicture}[>=stealth,yscale=1,xscale=1]  
\coordinate (A) at (2,4);
\coordinate (B) at (2,5.3);
\coordinate (C) at (3.3, 5.3);
\coordinate (D) at (3.3, 4);
\coordinate (E) at (6, 4); 
\coordinate (F) at (6, 3);
\coordinate (G) at (3.3, 3);
\coordinate (H) at (2, 3);
\coordinate (I) at (1.8, 4.7) ;
\fill[color=Frangipane] (A) -- (B) -- (C) -- (D) -- (E) -- (F) -- (G) -- (H) -- cycle;
\draw[color=red] (A) -- (B) ;
\draw[color=green](B) -- (C) ;
\draw[color=red] (C) -- (D) ;
\draw[color=blue] (H) -- (A) ;
\draw (D) -- (E) ;
\draw[color=blue] (E) -- (F);
\draw[color=black] (F) -- (G);
\draw[color=green] (G) -- (H);
\node at (1.8, 4.7) {$\lambda$};
\node at (1.8, 3.5) {$1$} ;
\node at (2.7, 5.5) {$\lambda$} ;
\node at (4.5, 4.2) {$b - \lambda$} ;
\end{tikzpicture}
\caption{Translation surface $L(b, e)$.}
\label{Lshaped}
\end{figure}
\begin{theorem}[McMullen \cite{MCM05}]
The Teichm\"uller curve generated by $L(b, e)$ is a Weierstrass curve of discriminant $D=e^2 + 4b$. 
When $D>9$ is such that $D\equiv 1\mod 8$, the curves generated by $L(b,e)$ and $L(b, -e)$ do not have the same spin invariant.
\end{theorem}
We use this explicit description to prove \cref{lemmaNoAuto} for Weierstrass curves. Recall that this proves \cref{corEchoes} for $V$ a Weierstrass curve that is not $W_9$.
\begin{proof}
Suppose that $\lambda \neq 1$.
The saddle connections on $L$ lift on a double cover $\tilde L$ to saddle connections with the same periods.
There are two saddle connections on $L$ with period $i$. 
Of those two, only one has the property that the saddle connection defined by turning with an angle $\pi$ from its endpoint has period $i$. 
Indeed, turning by an angle $\pi$ at the endpoint of the other will give a saddle connection with period $i\lambda$.
Namely it is the saddle connection going up from the only conical point in the interior of the bottom edge of $L$.
Therefore $\tilde L$ has exactly two saddle connections satisfying this property.
A translation of $\tilde L$ must either fix them, or interchange them.
Hence the group of translations of $\tilde L$ is of order two.
We have $\lambda = 1$ if and only if $2 = e + \sqrt{e^2 + 4b}$, that is when $b = 1 - e$.
We have assumed that $e+1 < b$, so the only possibility is $b=2$ and $e = -1$. The curve generated by $L(2, -1)$ is $W_9$.
\end{proof}
We now prove \cref{corEchoes} for $V = W_9$.
\begin{proof}
The curve $W_9$ is generated by an origami $\mathcal O$ with three squares.
We check directly that among the fifteen lifts of $\mathcal O$, exactly one has a group of translation of order larger than two.
The action of $\mathrm{Aff}^+(\mathcal O)$ on covers $S\to \Sigma$ preserves the order of the translation group of the associated lift, hence fixes this cover. Moreover, for the other covers, there are only two translations and we can apply \cref{sameDisc}.
\end{proof}
\subsection{Marking and basis of homology}
Let us first fix $(a_1, b_1, a_2, b_2)$, a basis for $H_1(L, \mathbb Z)$. 
We take for $a_1$ and $b_1$ respectively the horizontal and vertical sides of the top square of side $\lambda$. 
For $a_2$ and $b_2$, we take the bottom horizontal side of $L$ and the vertical (blue) side at the right of $L$, respectively.
We orient $a_1$ and $a_2$ towards the right, while $b_1$ and $b_2$ are oriented as going up.
This basis is symplectic.
We fix an identification $\Sigma\simeq L$.
We will use Dehn twists in cylinders decompositions of $L$ as in \cref{DehnTwistAnd} to study the action of its affine group on homology.
\subsection{Action on homology of Dehn twists}
Suppose that $s$ is a direction that decomposes $L$ into two cylinders $C_1$ and $C_2$.
Their moduli $m_1$ and $m_2$ are commensurable, since $L$ is a Veech surface.
Let $m$ be the lowest common multiple of $\frac{1}{m_1}$ and $\frac{1}{m_2}$ and write
$m\cdot m_1 = {k_1}$ and $m \cdot m_2 = {k_2}$. Since the action on homology modulo two of a Dehn twist in a cylinder is of order two, the action of the Dehn twist of \cref{DehnTwistAnd} only depends on the parity of $k_1$ and $k_2$.
Namely, the action is the same as the one of the Dehn twist in $C_1$ (resp. $C_2$) if $k_1$ is odd and $k_2$ is even (resp. $k_1$ is even and $k_2$ is odd). If both $k_1$ and $k_2$ are odd, then the action on homology is the same as the one of the composition of a Dehn twist in $C_1$ and a Dehn twist in $C_2$. Observe that we have $\frac{m_1}{m_2} = \frac{k_1}{k_2}$.
\subsection{Horizontal cylinder decomposition}
Let us compute the moduli of the cylinders that are in the horizontal direction.
The top square of side $\lambda$ has modulus $m_1=1$. The rectangle below has modulus $m_2 =\frac{1}{b}$.
Thus the affine group of $L(b, e)$ contains the map that is a right-Dehn twist in the bottom rectangle, and the $b$-th power of a right-Dehn twist in the top square.
Its action on homology gives the following element of the monodromy group in our chosen basis:
\[H = \begin{pmatrix}
 1 & b & 0 & 0\\
 0 & 1 & 0 & 0\\
 0 & 0 & 1 & 1\\
 0 & 0 & 0 & 1
\end{pmatrix}.
\]
\subsection{Vertical cylinder decomposition}
The vertical direction gives a decomposition into two cylinders. The cylinder on the left has modulus $m_1 = \frac{\lambda}{\lambda+1}$, while the second has modulus $m_2 = {b-\lambda}$.
Since $\lambda^2 = e\lambda + b$, their ratio is
\[\frac{m_2}{m_1} = \frac{\left (b-\lambda\right )\left (\lambda+1\right )}{\lambda} = \frac{b\lambda + b - \lambda^2 - \lambda}{\lambda} = b-e - 1.
\]
Similarly to the horizontal decomposition, we obtain an element of the monodromy group: 
\[V = 
\begin{pmatrix}
 1 & 0 & 0 & 0\\
 1 & 1 & 1 & 0\\
 0 & 0 & 1 & 0\\
 1 & 0 & b-e & 1
\end{pmatrix}.
\]
We can obtain the third column of $V$ by remarking that the image by the Dehn twist of $(-1, 0, 1, 0)$ is $(-1, 0, 1, b-e-1)$.
\subsection{Rotation}
We want to prove that the monodromy group modulo two strictly contains the one generated by $H$ and $V$, when $D\equiv 1\mod 8$.
\begin{lemma}\label{D1}
If $D\equiv 1\mod 8$, then the monodromy group modulo two of $L$ is not generated by the horizontal and vertical twists.
\end{lemma}
Recall that $D = e^2 + 4b$. We thus assume for the rest of the section that $b$ is even and $e\in \{-1, 1\}$.
\subsubsection{The case $e=-1$}
There exists a simple argument to prove \cref{D1} when $e = -1$.
In this case, the following matrix belongs to the Veech group of $L$
\[\begin{pmatrix}
    \frac{b}{\lambda+1} & 0\\
    0 & \frac{1}{\lambda}
\end{pmatrix}
\begin{pmatrix}
    0 && 1\\
    -1 && 0
\end{pmatrix}.
\]
Indeed, the rotation of order four applied to $L$ gives \cref{LshapedRotated}.
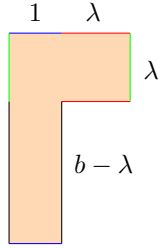
\begin{figure}[h!]
\begin{tikzpicture}[>=stealth,yscale=0.7,xscale=0.7]  
\coordinate (A) at (4, -2);
\coordinate (B) at (5.3, -2);
\coordinate (C) at (5.3, -3.3);
\coordinate (D) at (4, -3.3);
\coordinate (E) at (4, -6); 
\coordinate (F) at (3, -6);
\coordinate (G) at (3, -3.3);
\coordinate (H) at (3, -2);
\coordinate (I) at (4.7, -1.8) ;
\fill[color=Frangipane] (A) -- (B) -- (C) -- (D) -- (E) -- (F) -- (G) -- (H) -- cycle;
\draw[color=red] (A) -- (B) ;
\draw[color=green](B) -- (C) ;
\draw[color=red] (C) -- (D) ;
\draw[color=blue] (H) -- (A) ;
\draw (D) -- (E) ;
\draw[color=blue] (E) -- (F);
\draw[color=black] (F) -- (G);
\draw[color=green] (G) -- (H);
\node at (4.6, -1.6) {$\lambda$};
\node at (3.5, -1.6) {$1$} ;
\node at (5.7, -2.7) {$\lambda$} ;
\node at (4.8, -4.5) {$b - \lambda$} ;
\end{tikzpicture}
\caption{Translation surface $L(b, e)$ rotated.}
\label{LshapedRotated}
\end{figure}
Then the rescaling of the axes changes the bottom rectangle to a new one of sides 
$\frac{b-\lambda}{\lambda}$ and $\frac{b}{\lambda+1}$. But since $b = \lambda^2 + \lambda$, it becomes a square of side $\lambda$. Similarly, the top rectangle changes to a new one of width $b$ and height $1$.
A corresponding homeomorphism in the Veech group exchanges the two handles $(a_1, b_1)$ and $(a_2, b_2)$. Its action on $H_1(\Sigma, \mathbb Z/2\mathbb Z)$ will send $v = (0, 1, 0, 0)$ to $(-1, 0, 1, 0)$, while $v$ is fixed by both the $H$ and $V$. This proves \cref{D1} when $e = -1$.
\subsubsection{The case $e=1$}
We suppose for the rest of the section that $e=1$. The $\mathrm{GL}_2^+(\mathbb R)$-orbit of $L(b, e)$ contains the $L$-shaped translation surface $\mathcal L(b, e)$, that is the same as $L(b, e)$, but with the top square of side $\lambda - 2$ and the bottom rectangle of width $b-2$.
Indeed, apply the following matrix to $L(b, e)$:
\[
\begin{pmatrix}
    \frac{b-\lambda}{\lambda} & 0\\
    0 & \frac{1}{\lambda}
\end{pmatrix}
\begin{pmatrix}
    0 && 1\\
    -1 && 0
\end{pmatrix}.
\]
After rotation and rescaling, the top rectangle has width 
\[\frac{(b-\lambda)(\lambda +1)}{\lambda} = b-e-1 = b-2.\]
The remaining rectangle is a square of side $\frac{b-\lambda}{\lambda} = \lambda - e - 1 = \lambda - 2$.
We mark the surface $\mathcal L(b, e)$ similarly to $L(b, e)$.
\subsection{Diagonal cylinder decompositions}
Our goal is to prove \cref{D1} in the remaining case $e=1$.
\subsubsection{The case $b\equiv 2\mod 4$}
Let us start by assuming that $b$ is not a multiple of four.
Let us consider the cylinder decomposition associated with the slope $s=\frac{2}{b}$.
We assume for now that $\frac{b}{2} > \lambda$. 
This is equivalent to $b > 6$. 
This requirement allows us to have a simple picture for the cylinder decomposition, as shown in \cref{slope}.
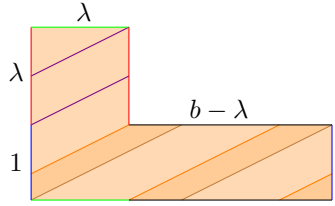
\begin{figure}[h!]
\begin{tikzpicture}[>=stealth,yscale=1,xscale=1]  
\coordinate (A) at (2,4);
\coordinate (B) at (2,5.3);
\coordinate (C) at (3.3, 5.3);
\coordinate (D) at (3.3, 4);
\coordinate (E) at (6, 4); 
\coordinate (F) at (6, 3);
\coordinate (G) at (3.3, 3);
\coordinate (H) at (2, 3);
\coordinate (I) at (1.8, 4.7) ;
\coordinate (J) at (4, 4);
\coordinate (K) at (4, 3);
\coordinate (L) at (3.3, 4.65);
\coordinate (M) at (2, 4.65);
\coordinate (N) at (5.3 , 4);
\coordinate (O) at (5.3, 3);
\coordinate (P) at (6, 3.35);
\coordinate (Q) at (2, 3.35);
\fill[color=Frangipane] (A) -- (B) -- (C) -- (D) -- (E) -- (F) -- (G) -- (H) -- cycle;
\fill[color=Peach] (Q) -- (D) -- (J) -- (H) -- cycle ;
\fill[color=Peach] (G) -- (N) -- (E) -- (K) -- cycle ;
\fill[color=Peach] (O) -- (P) -- (F) -- cycle ;
\draw[color=brown] (K) -- (E) ;
\draw[color = brown] (H) -- (J) ;
\draw[color=violet] (A) -- (L) ;
\draw[color=violet] (M) -- (C) ;
\draw[color = orange] (G) -- (N) ;
\draw[color=orange] (O) -- (P) ;
\draw[color=orange] (Q) -- (D) ;
\draw (D) -- (E) ;
\node at (1.8, 4.7) {$\lambda$};
\node at (1.8, 3.5) {$1$} ;
\node at (2.7, 5.5) {$\lambda$} ;
\node at (4.5, 4.2) {$b - \lambda$} ;
\draw[color=blue] (E) -- (F);
\draw[color=black] (F) -- (G);
\draw[color=green] (G) -- (H);
\draw[color=red] (A) -- (B) ;
\draw[color=green](B) -- (C) ;
\draw[color=red] (C) -- (D) ;
\draw[color=blue] (H) -- (A) ;
\end{tikzpicture}
\caption{Cylinder decomposition with slope $s=2/b$.}
\label{slope}
\end{figure}
The elements $\alpha = (0, 0, 1, 2)$ and $\beta = (\frac{b}{2}, 1, 0, 0)$ of $H_1(\Sigma, \mathbb Z)$ represent saddle connections with slope $s$, see the brown and violet lines in \cref{slope}. Let $\gamma = (1, 0, 0, 1)$ and $\delta = (1, 0, 0, 0)$. The elements $(\alpha, \gamma)$ form a basis of the homology of the handle formed by the smaller cylinder $C_1$, and $(\alpha + \beta, \delta)$ form a basis of the homology of the handle formed by the longer cylinder $C_2$.
Let us compute the moduli of these cylinders.
Let us start with the cylinder $C_1$.
The period of $\alpha$ is $b + 2i$, while the period of $\gamma$ is $\lambda + i$.
We wish to compute the modulus of the cylinder defined by these complex numbers.
Multiplying both numbers by a complex number does not change the modulus of the cylinder they form, hence we can consider the parallelogram formed by $b^2 + 4$ and $(\lambda + i)(b - 2i)$. Thus the modulus $m_1$ of $C_1$ is:
\[m_1 = \frac{b - 2\lambda}{b^2+4}.\]
The period of $\alpha + \beta$ is $\frac{b}{2}\lambda + b + i(\lambda + 2)$, and the period of $\delta$ is $\lambda$.
Thus the modulus $m_2$ of $C_2$ is:
\[m_2 = \frac{\lambda(\lambda+2)}{(\frac{b}{2}\lambda + b)^2 + (\lambda+2)^2} = \frac{4\lambda}{(b^2+4)({\lambda} + 2)}.\]
Since $\lambda^2 = e\lambda + b$, we have
\[\frac{m_1}{m_2} = \frac{(b-2\lambda)(\lambda + 2)}{4\lambda} = \frac{b\lambda + 2b - 2\lambda^2- 4\lambda}{4\lambda} = \frac{\frac{b}{2} - e - 2}{2}.\]
Since $b$ is not a multiple of four, $\frac{b}{2} - e - 2$ is even, and the twist associated with the slope $s$ will send $\delta$ to $\delta + \alpha + \beta$ in the monodromy modulo two group. Therefore, the action of this group on $H_1(\Sigma, \mathbb Z/2\mathbb Z)$ must have $(1, 0, 0, 0)$ in the orbit of $(1, 1, 1, 0)$. This is not done by any element of the group generated by the horizontal and vertical twists. 
Indeed the vector $(1, 1, 1, 0)$ is fixed by both $V$ and $H$.
We thus have proven \cref{D1} when $b>6$ is not a multiple of four.
\subsubsection{The case $b\equiv 0\mod 4$}
We now turn to the case where $b$ is a multiple of four.
To deal with this case, we will for convenience use the model $\mathcal L(b, e)$ instead of $L(b, e)$.
Since they are in the same $\mathrm{GL}^+_2(\mathbb R)$-orbit, their monodromy groups modulo two are isomorphic.
The horizontal and vertical cylinder decomposition give the same matrices $H$ and $V$. Indeed when computing the moduli of the associated cylinders, we obtain $\frac{m_1}{m_2} = b-2$ for the horizontal one, and for the vertical one
\[
\frac{m_2}{m_1} = \frac{(b-\lambda)(\lambda-1)}{\lambda - 2} = \frac{b\lambda - 2b - e\lambda + \lambda}{\lambda - 2} = b.
\]
We consider, similarly to the previous diagonal cylinder decomposition, the slope $s = \frac{2}{b-2}$.
We obtain a picture similar to \cref{slope}, provided $\lambda - 2 < \frac{b-2}{2}$. 
This is actually always the case: the inequality is satisfied as soon as $b > 2$, and $b$ is a multiple of four. 
Let us compute the moduli of the cylinders $C_1$ and $C_2$. We have as in the previous case, setting $b' = b-2$ and $\lambda' = \lambda - 2$:
\[\frac{m_1}{m_2} = \frac{(b'-2\lambda')(\lambda' + 2)}{4\lambda'} = \frac{(b-2\lambda + 2)\lambda}{4(\lambda - 2)} = \frac{b\lambda - 2(e\lambda + b) + 2\lambda}{4(\lambda - 2)} = 
\frac{b}{4}.\]
Since $b$ is a multiple of $4$, as before the twist in the cylinder $C_2$ will create an element of the monodromy group that is not in the group generated by $H$ and $V$. We thus have proven \cref{D1} when $b$ is a multiple of four.
\subsubsection{The case $b=6$}
The only cases remaining are when $b \leqslant 6$ is not a multiple of four. 
Note that when $b=2$, since $e+1 < b$, we have $e=-1$. This corresponds to the fact that there is only one Weierstrass curve of discriminant $D=9$ by \cite[Theorem 1.1]{MCM05}.
Hence we just need to consider the case $b=6$. 
This case corresponds to a square-tiled surface, tiled by five squares.
Let us give two ways to prove \cref{D1} in this case.
One way, that gives a uniform approach to \cref{D1} but that is tedious, is to consider the cylinder decomposition associated with the slope $s = \frac{1}{2}$, as pictured in \cref{bis6}.
\begin{figure}[h!]
\begin{tikzpicture}[>=stealth,yscale=0.8,xscale=0.8]  
\coordinate (A) at (0,0);
\coordinate (B) at (3,0);
\coordinate (C) at (6, 0);
\coordinate (D) at (6, 1);
\coordinate (E) at (3, 1); 
\coordinate (F) at (3, 4);
\coordinate (G) at (0, 4);
\coordinate (H) at (0, 1);
\coordinate (I) at (3, 1.5) ;
\coordinate (J) at (0, 1.5) ;
\coordinate (K) at (3, 3);
\coordinate (L) at (0, 3) ;
\coordinate (M) at (2, 4);
\coordinate (N) at (2, 0);
\coordinate (O) at (4, 1);
\coordinate (P) at (4, 0);
\coordinate (Q) at (6, 1);
\coordinate (R) at (0, 1) ;
\coordinate (S) at (3, 2.5) ;
\coordinate (T) at (0, 2.5) ;
\coordinate (U) at (3, 4) ;
\coordinate (V) at (3, 0) ;
\coordinate (W) at (5, 1) ;
\coordinate (X) at (5, 0) ;
\coordinate (Y) at (6, 0.5) ;
\coordinate (Z) at (0, 0.5) ;
\coordinate (AA) at (3, 2) ;
\coordinate (BB) at (0 ,2) ;
\coordinate (CC) at (3, 3.5) ;
\coordinate (DD) at (0, 3.5) ;
\coordinate (EE) at (1, 4) ;
\coordinate (FF) at (1, 0) ;
\coordinate (GG) at (3, 1) ;
\fill[color=Frangipane] (A) -- (B) -- (C) -- (D) -- (E) -- (F) -- (G) -- (H) -- cycle;
\fill[color=Peach] (A) -- (I) -- (AA) -- (Z) -- cycle;
\fill[color=Peach] (J) -- (K) -- (CC) -- (BB) -- cycle;
\fill[color=Peach] (L) -- (M) -- (EE) -- (DD) -- cycle;
\fill[color=Peach] (N) -- (O) -- (GG) -- (FF) -- cycle;
\fill[color=Peach] (P) -- (Q) -- (W) -- (V) -- cycle;
\fill[color=Peach] (X) -- (Y) -- (C) -- cycle;
\draw[color=orange] (V) -- (W) ;
\draw[color=orange] (X) -- (Y) ;
\draw[color=orange] (Z) -- (AA) ;
\draw[color=orange] (BB) -- (CC) ;
\draw[color=orange] (DD) -- (EE) ;
\draw[color=orange] (FF) -- (GG) ;
\draw[color=brown] (A) -- (I) ;
\draw[color=brown] (J) -- (K) ;
\draw[color=brown] (L) -- (M) ;
\draw[color=brown] (N) -- (O) ;
\draw[color=brown] (P) -- (Q) ;
\draw[color=green] (A) -- (B) ;
\draw[color=black](B) -- (C) ;
\draw[color=blue] (C) -- (D) ;
\draw[color=black] (D) -- (E) ;
\draw[color=red] (E) -- (F);
\draw[color=green] (F) -- (G);
\draw[color=red] (G) -- (H);
\draw[color=blue] (H) -- (A);
\draw[color=violet] (R) -- (S) ;
\draw[color=violet] (T) -- (U) ;

\end{tikzpicture}
\caption{Cylinder decomposition of  $L(6, 1)$ with slope $s=1/2$.}
\label{bis6}
\end{figure}
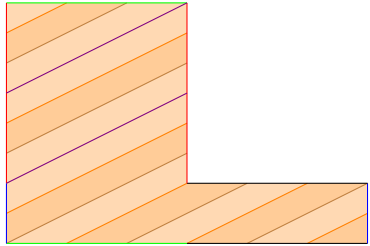
We leave it to the reader to check that this cylinder decomposition has two cylinders $C_1$ and $C_2$ and that both a twist in $C_1$ and $C_2$ send elements in $H_1(\Sigma, \mathbb Z/2\mathbb Z)$ to others that are not in their orbit under the action of $\langle H, V\rangle $.

An easier approach consists in remarking that, since $L(6, 1)$ is square-tiled, we may act by $\mathrm{SL}_2(\mathbb Z)$ on it to find elements in its Veech group directly.
We will work with $\mathcal O = \begin{pmatrix}
    \frac{1}{3} & 0\\
    0 & 1
\end{pmatrix}\cdot L(6, 1)$ instead of $L(6, 1)$ since it is then tiled by five unit squares.
We transport our basis of homology to $\mathcal O$ with this rescaling of the $x$-axis.
The permutations in $\mathfrak S_5$ defining $\mathcal O$ can be taken to be $\alpha = (1, 2)$ and $\beta = (2,3,4,5)$.
Let  $L = \begin{pmatrix}
    1 & 0\\
    1 & 1
\end{pmatrix}$ and $R = \begin{pmatrix}
    1 & 1\\
    0 & 1
\end{pmatrix}$.
The matrices $L$ and $R$ transform $(\alpha, \beta)$ respectively to $(\alpha\beta^{-1}, \beta)$ and $(\alpha, \beta \alpha^{-1})$, see \textit{e.g.} \cite[Section 2.3]{M22}.
This perspective allows us to find an element in the monodromy group modulo two that is not in $\langle H, V\rangle$.

\begin{lemma}\label{VeechL}
The matrix $M = L R^3 L$ belongs to the Veech group of $\mathcal O$ and the element in the monodromy group modulo two it induces does not leave $b_1$ invariant.
\end{lemma}
In \cref{VeechL}, we can talk about \emph{the} element of the monodromy group induced by $M$ since $\mathcal O$ does not have any translation, hence $\mathrm{SL}(\mathcal O) = \mathrm{Aff}^+(\mathcal O)$.
This is actually the case for every square-tiled surface in $\mathcal H(2g-2)$, see for example \cite[Proposition 39]{M22}. Concretely the matrix $M$ is
\[M = \begin{pmatrix}
    4 & 3\\
    5 & 4
\end{pmatrix}.
\]
\begin{proof}
Let us first show that $LR^3L$ is in the Veech group of $\mathcal O$.
The matrix $L$ changes $(\alpha, \beta)$ to $(\gamma, \beta)$ where $\gamma = \alpha\beta^{-1} = (1\ 5\ 4\ 3\ 2)$. Then $R^3$ sends $(\gamma, \beta)$ to $(\gamma, \delta)$ where $\delta = \beta\gamma^{-3} = (1\ 4\ 3\ 2)$. Finally $L$ sends $(\gamma, \delta)$ to $(\epsilon, \delta)$ where $\epsilon = \gamma\delta^{-1} = (1 \ 5)$.
The pair $(\epsilon, \delta)$ is conjugated to $(\alpha, \beta)$ by $(1\ 2\ 5)(3\ 4)$. Therefore $M\in \mathrm{SL}(\mathcal O)$.
In order to show that $b_1$ is not stabilised by $M$, it suffices by \cref{veech_cover} to show that the cover $\tilde {\mathcal O}$ of $\mathcal O$ associated with $b_1$ does not contain $M$ in its Veech group. The cover associated with $b_1$ has the form of the right side of \cref{twoExamples}.
The permutations defining it can be taken to be $\alpha' = (1\ 2)(6\ 7)(3\ 8)(4\ 9)(5\ 10)$ and $\beta' = (2\ 3\ 4\ 5)(7\ 8\ 9\ 10)$. 
Applying the same transformations as before, one obtains
\[(\epsilon', \delta') = \left ((1\ 10)(5\ 6), (1\ 4\ 8\ 2)(3\ 7\ 6\ 9)(5\ 10)\right ).\]
The pair $(\epsilon', \delta')$ is obviously not conjugated to $(\alpha', \beta')$ thus $M\notin \mathrm{SL}(\tilde{\mathcal O})$.
\end{proof}
This completes the proof of \cref{D1}.
\section{Monodromy of the decagon}\label{sectionDec}

In \cite{MCM06}, McMullen established that the regular decagon $D$ is the only abelian differential in $\mathcal H(1, 1)$ that generates a primitive Teichmüller curve, up to the action of $\mathrm{GL}_2^+(\mathbb R)$. Let us fix an identification $\Sigma\simeq D$.
In this section we compute the monodromy group of $D$, that is the image of the monodromy map
\[\rho\colon \mathrm{Aff}^+(D)\to \mathrm{Sp}(H_1(\Sigma, \mathbb Z)).\]
Since the translation group of the decagon $\mathrm{Tr}(D)$ is trivial, we identify the affine group with the Veech group $\mathrm{SL}(D)$ of $D$. The seminal article of Veech \cite{V89} describes the Veech groups of regular polygons.
In order to recall this description in our case, let us denote by $R$ the matrix of the rotation of angle $\frac{\pi}{5}$, and by $T$ the following matrix
\[
T = \begin{pmatrix}
    1 & 2\cot \frac{\pi}{10}\\
    0 & 1
\end{pmatrix} = \begin{pmatrix}
    1 & 2\sqrt{5 + 2\sqrt 5}\\
    0 & 1
\end{pmatrix}.
\]
It follows from \cite[Theorem 1.1]{V89} that the Veech group of the decagon is generated by the matrices $R$ and $T$.
Let us fix a basis $(a_1, b_1, a_2, b_2)$ of $H_1(\Sigma, \mathbb Z)$, as pictured in \cref{FigDecagone}, where the curves are oriented from left to right.
\begin{figure}[h!]
\begin{tikzpicture}[>=stealth,yscale=1.8,xscale=1.8]  
\coordinate (A) at (1,0);
\coordinate (B) at (0.8090,0.5878);
\coordinate (C) at (0.3090, 0.9511);
\coordinate (D) at (-0.3090, 0.9511);
\coordinate (E) at (-0.8090, 0.5878); 
\coordinate (F) at (-1, 0);
\coordinate (G) at (-0.8090, -0.5878);
\coordinate (H) at (-0.3090, -0.9511);
\coordinate (I) at (0.3090, -0.9511) ;
\coordinate (J) at (0.8090, -0.5878) ;
\fill[color=Frangipane] (A) -- (B) -- (C) -- (D) -- (E) -- (F) -- (G)  -- (H) -- (I) -- (J) -- cycle;
\draw[color=red] (A) -- (B) ;
\draw[color=blue] (B) -- (C) ;
\draw[color= green] (C) -- (D) ;
\draw[color = yellow] (D) -- (E) ;
\draw[color = brown] (E) -- (F) ;
\draw[color = red] (F) -- (G) ;
\draw[color = blue] (G) -- (H) ;
\draw[color = green] (H) -- (I);
\draw[color = yellow] (I) -- (J) ;
\draw[color = brown] (J) -- (A) ;
\coordinate (K) at (0.9045, 0.2939) ;
\coordinate (L) at (-0.9045, -0.2939) ;
\draw[color = red] (K) -- (L) ;
\coordinate (M) at (-0.9045, 0.2939);
\coordinate (N) at (0.9045, -0.2949);
\draw[color = blue] (M) -- (N) ;
\draw[color = red] (E) -- (C) ;
\draw[color = blue] (D) -- (B) ;
\node at (-0.603, -0.35) {$a_1$};
\node at (-0.603, 0.05) {$b_1$};
\node at (-0.603, 0.50) {$a_2$};
\node at (0.25, 0.61) {$b_2$} ;
\end{tikzpicture}
\caption{The decagon $D$.}
\label{FigDecagone}
\end{figure}
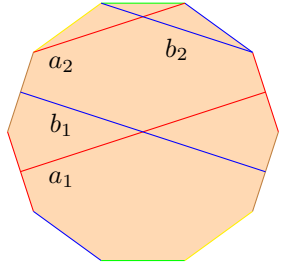

Before computing the monodromy of $D$, let us prove \cref{lemmaNoAuto}.
\begin{proof}
There is only one saddle connection on $D$ with period $2$.
There are thus two saddle connections on any cover of degree two $\tilde D$ with period $2$.
Any translation must either fix them, or interchange them.
The translation group of $\tilde D$ is thus of order two.
\end{proof}
This finishes the proof of \cref{corEchoes}. We now turn to the computation of the monodromy.
Let $\xi = e^{i\pi/5}$, and $\alpha_1, \beta_1, \alpha_2, \beta_2$ be the periods of $a_1, b_1, a_2, b_2$, respectively.
We can compute these numbers directly:
\begin{align*}
    \alpha_1 &= \xi + 1 = \frac{5 + \sqrt{5}}{4} + i\sqrt{\frac{5 - \sqrt 5}{8}}\\
    \beta_1 &= \overline{\alpha_1} = \frac{5 + \sqrt 5}{4} - i\sqrt{\frac{5 - \sqrt 5} 8}\\
    \alpha_2 &=  \xi^2 - \xi^4 = \frac {\sqrt{5}} 2 + i\frac{\sqrt{5 - 2\sqrt 5}} 2\\
    \beta_2 &= \xi - \xi^{3} = \frac{\sqrt 5}{2} - i\frac{\sqrt{5 - 2\sqrt 5}}{2}.
\end{align*}
The numbers $\alpha_1$, $\beta_1$, $\alpha_2$, $\beta_2$ are linearly independent.
\begin{lemma}\label{perInjective}
The period map $\mathrm{Per}\colon H_1(\Sigma, \mathbb Z)\to \C$ is injective.
\end{lemma}
\begin{proof}
    Suppose $(x_1, y_1, x_2, y_2)\in \mathbb Z^4$ is such that 
    \[x_1\alpha_1 + y_1\beta_1 + x_2\alpha_2 + y_2\beta_2 = 0.\]
The real part gives $x_1 + y_1 = 0$ and $x_2 + y_2 = 0$. The imaginary part then gives
\[x_1^2(5 - \sqrt{5}) = 2x_2^2(5 - 2\sqrt 5).\]
Therefore $x_1 = x_2 = 0$ and $y_1=y_2 = 0$.
\end{proof}
\begin{rmk}
\cref{perInjective} also follows from the fact that $\mathcal D$ is algebraically primitive, see \cite[Lemma 3.2]{FL23}.
\end{rmk}
This will allow us to derive the monodromy representation from the action of the Veech group on periods. Indeed, the Veech group $\mathrm{SL}(D)$ acts on $\C = \mathbb R^2$, and on periods since for $A\in \mathrm{SL}(D)$ and $\gamma\in H_1(\Sigma, \mathbb Z)$,
\[A\cdot \mathrm{Per}(\gamma) = \mathrm{Per}(\rho(A)\cdot \gamma).
\]
Since $R$ is a rotation of angle $\frac{\pi} 5$, we have $R\beta_1 = \xi\beta_1 = \alpha_1$, and $R\beta_2 = \xi \beta_2 = \alpha_2$.
Hence $\rho(R)$ sends $b_i$ to $a_i$ for $1\leqslant i \leqslant 2$.
From $\xi^4 = -\xi^{-1}$, we get
\[
R\alpha_1 = \xi (\xi + 1) = \xi^2 + \xi = (\xi + 1) - (\xi^{-1} + 1) + (\xi^2 - \xi^4).
\]
Therefore $R\alpha_1 = \alpha_1 - \beta_1 + \alpha_2$.
Similarly, 
\[
R\alpha_2 = \xi(\xi^2 - \xi^4) = \xi^3 +1 = (\xi + 1) - (\xi - \xi^3).
\]
Thus $R\alpha_2 = \alpha_1 - \beta_2$. We can now write the matrix of $\rho(R)$ in the basis $(a_1, b_1, a_2, b_2)$
\[
\rho(R) = 
\begin{pmatrix}
    1 && 1 && 1 && 0\\
    -1 && 0 && 0 && 0\\
    1 && 0 && 0 && 1\\
    0 && 0 && -1 && 0
\end{pmatrix}.
\]
Let us now turn to the computation of $\rho(T)$.
Straightforward computations give

\begin{align*}
T\alpha_1 &= 3\cdot\frac{5+\sqrt 5}{4} + i\sqrt{\frac{{5 - \sqrt 5}}{8}}\\
T\beta_1 &= -\frac{5 + \sqrt 5}{4} - i\sqrt{\frac{5 - \sqrt 5}{8}}\\
T\alpha_2 &= 3 \frac{\sqrt 5}{2}+ i\frac{\sqrt{5 - 2\sqrt 5}}{2}\\
T\beta_2 &= - \frac{\sqrt 5}{2} - \frac{\sqrt{5 - \sqrt 5}}{2}.
\end{align*}
Therefore, we have $T\alpha_1 = 2\alpha_1 + \beta_1$, $T\beta_1 = -\alpha_1$, $T\alpha_2 = 2\alpha_2 + \beta_2$, and $T\beta_2 = -\alpha_2$.
We now have the image of $T$ by the monodromy map:
\[
\rho(T) = 
\begin{pmatrix}
    2 && -1 && 0 && 0\\
    1 && 0 && 0 && 0\\
    0 && 0 && 2 && -1\\
    0 && 0 && 1 && 0
\end{pmatrix}.
\]
Reducing modulo two, we easily see that the monodromy group modulo two of $D$ is dihedral of order ten.
One could alternatively use results of \cite{MCM23_2} to describe this action.
We are now set to prove \cref{thDecagone}: $\mathcal D$ has three echoes in $\mathcal M_3$.
\begin{proof}
By \cref{corEchoes}, it suffices to describe the action of the monodromy group on $H_1(\Sigma, \mathbb Z/2\mathbb Z)$.
We form the graph whose vertices are non-zero elements of $H_1(\Sigma, \mathbb Z/2\mathbb Z)$, with an edge between each vertex $v$ and $Mv$ for $M\in\{\rho(R), \rho(T)\}$. We verify that this graph has three connected components, each of size five. 
\end{proof}
Note that each of the three echoes is a cover of $\mathcal D$ of degree five.
We deduce from this analysis \cref{corRelDec}.
\begin{proof}
Let $(X, \omega)$ be a genus three Veech surface covering $(X_0, \omega_0)\in \mathcal H(1, 1)$. If $(X, \omega)$ is not covering a surface of genus one, then $\mathrm{SL}(X, \omega)$ is a subgroup of $\mathrm{SL}(X_0, \omega_0)$, by \cite[Corollary 2.2]{MCMPrym}. Thus $(X_0, \omega_0)$ is a Veech surface in $\mathcal H(1, 1)$, generating a primitive Teichm\"uller curve. It follows from \cite{MCM06} that $(X_0, \omega_0)$ is, after rescaling, in the $\mathrm{SL}_2(\mathbb R)$-orbit of $D$.
Therefore $(X, \omega)$ belongs to one of the three Teichm\"uller discs generated by lifts of $D$.
\end{proof}
\subsubsection{Higher genera}
Our techniques can be applied to study echoes of $\mathcal D$ in higher genera moduli spaces. One can for example replace degree two covers by cyclic covers of a given degree $n$, that is covers associated with a subgroup of the form $\ker(\pi_1(\Sigma)\twoheadrightarrow \mathbb Z/n\mathbb Z)$.
Let us call the \emph{cyclic} echoes of $\mathcal D$ the Teichm\"uller curves obtained by pulling back $D$ with these covers.
The cyclic echoes of $\mathcal D$ lie in $\mathcal M_{n+1}$. 
We can mimic the proof of \cref{lemmaNoAuto}, and get that the echoes are still classified by the action of the affine group on the covers as in \cref{corEchoes}, see also \cref{transG} in \cref{sectionHigher}.
We compute the graphs as before, where this time the vertices are the primitive elements of $H_1(\Sigma, \mathbb Z/n\mathbb Z)$. 
We can write a program that counts the connected components of these graphs, and get the number $N(n)$ of cyclic echoes of $\mathcal D$ in $\mathcal M_{n+1}$ for small values of $n$, see \cref{tableDec}.
\begin{center}    \begin{table}[!h]
\begin{tabular}{|c ||c | c | c|c|c| c| c|c|c|c|c|c|c|c|}
    \hline
     $n = g-1$ & $2$ & $3$ & $4$ & $5$  & $6$ & $7$ & $8$ & $9$ & $10$ & $11$ & $12$ & $13$ & $14$ & $15$\\
     \hline
     $N(n)$ & $3$ & $1$ & $3$ & $8$ & $3$ & $1$ & $3$ & $1$ & $24$ & $3$ & $3$ & $1$ & $3$ & $8$\\
     \hline
\end{tabular}
\caption{The numbers of cyclic echoes of $\mathcal D$ in $\mathcal M_g$.}
\label{tableDec}
\end{table}
\end{center}
\section{Real multiplication and monodromy groups}\label{sectionReal}
In this section, we prove \cref{commutationReal}. 
We first give a general proof, and then for the specific case of Weierstrass curves with non-square discriminant, we indicate an alternative and more direct approach.
\subsection{Theoretical approach}
\subsubsection{Commutation with real multiplication}
Let us fix $X\in \mathcal T(\Sigma)$ a complex structure on $\Sigma$.
We identify $\Omega(X)$ with $H^1(\Sigma, \mathbb R)$ by the map that sends $\alpha$ to $\mathrm{Re}(\int \alpha)$.
Recall that $[\alpha, \beta]\mapsto \frac{i}{2}\int \alpha\wedge \bar \beta$ is an inner product on $\Omega(X)$.
An endomorphism $T\colon H_1(\Sigma, \mathbb Z)\to H_1(\Sigma, \mathbb Z)$ extends to an $\mathbb R$-linear endomorphism of $H_1(\Sigma, \mathbb R)$.
Its dual $T^\vee\colon \Omega(X)\to\Omega(X)$ is a real linear endomorphism.
When $T^\vee$ is complex linear and self-adjoint, we can diagonalise it in an orthogonal basis.
\begin{lemma}\label{determinedPer}
Let $(X, \omega)\in \Omega \mathcal T(\Sigma)$, and  $T$ be a complex linear self-adjoint endomorphism of $H_1(\Sigma, \mathbb Z)$ that admits $\omega$ as eigenvector.
Its dual $T^\vee\colon \Omega(X)\to \Omega(X)$ is diagonal in any orthogonal basis of the form $(\omega, \omega')$.
\end{lemma}
\begin{proof}
Since $T$ is self-adjoint, $T^\vee$ also is, and preserves $\omega^\perp = \mathbb C \omega'$.
\end{proof}
Let us verify that an element $f\in \mathrm{Aff}^+(X, \omega)$ preserves $\mathbb C\omega$ and its orthogonal.
\begin{lemma}\label{preserves}
Let $(X, \omega)\in \Omega\mathcal T(\Sigma)$ and $f\in \mathrm{Aff}^+(X, \omega)$.
The endomorphism $\rho(f)^\vee\colon \Omega(X)\to \Omega(X)$ preserves both $\mathbb C \omega$ and its orthogonal.
\end{lemma}
\begin{proof}
Under our identification, $\mathbb C\omega = \mathbb R \mathrm{Re}(\int \omega) + \mathbb R\mathrm{Im}(\int \omega)$.
The endomorphism $\rho(f)^\vee\colon \Omega(X)\to \Omega(X)$ sends $\mathrm{Re} \int_\cdot \alpha$ to $\mathrm{Re}\int_{\rho(f)\cdot} \alpha$.
Since $f\in \mathrm{Aff}^+(X, \omega)$, there exists $A\in \mathrm{SL}_2(\mathbb R)$ such that $\int_{\rho(f)\cdot} \omega= A\int_\cdot \omega$.
In particular for $z\in \mathbb C$,
\[\rho(f)^\vee (z\omega) = \mathrm{Re}\left ( zA\int \omega\right )\in \mathbb R \mathrm{Re}(\int \omega) + \mathbb R \mathrm{Im}(\int \omega).\]
Therefore $\rho(f)^\vee$ preserves $\mathbb C\omega$.
Since $\rho(f)$ is symplectic, the adjoint of $\rho(f)^\vee$ is its inverse, and $\rho(f)^\vee$ preserves $\omega^\perp$.
\end{proof}

We are now able to prove \cref{commutationReal}. 
Let us give a slightly more general statement.
\begin{prop}\label{thCom}
Let $(X, \omega)\in \Omega\mathcal T(\Sigma)$ and $f\in \mathrm{Aff}^+(X, \omega)$. Any self-adjoint complex linear endomorphism $T\colon H_1(\Sigma, \mathbb Z)\to H_1(\Sigma, \mathbb Z)$ with real eigenvalues such that $\omega$ is an eigenvector of $T^\vee$ commutes with $\rho(f)$.
\end{prop}

\begin{proof}
Let $S = \rho(f)^{-1} T \rho(f)$.
Let us pick $(\omega, \omega')$ an orthogonal basis of $\Omega(X)$.
The spaces $\mathbb C\omega$ and $\mathbb C\omega'$ are composed of eigenvectors of $T$ by \cref{determinedPer}.
It follows from \cref{preserves} that $\rho(f)^\vee$ preserves both of them.
Therefore if $T^\vee(\omega) = \lambda \omega$ and $\rho(f)^{-1}(\omega) = \mu \omega$, shortening $S^\vee$ and $T^\vee$ by $S$ and $T$, we have
\[S(\omega) = \rho(f) T\rho (f)^{-1} (\omega) = \rho(f) T(\mu \omega) = \lambda \rho(f) (\mu \omega) = \lambda \omega.
\]
Note that we used that $\lambda$ is real.
Similarly, $S^\vee(\omega') = T^\vee(\omega')$.
Therefore $S = T$.
\end{proof}

\subsubsection{Real multiplication on $L$}\label{RealMulitplication}
Recall that $W_D\subset \mathcal M_2$ is the locus of Riemann surfaces whose Jacobian admits real multiplication by an order $\mathcal O_D$, and an eigenform $\omega$ with a single zero.
Let us find $T$ a generator of $\mathcal O_D$ for $L(b, e)$ the surface \cref{Lshaped} of \cref{explicit}.
We fix again $\lambda = \frac{e + \sqrt{e^2 + 4b}}{2}$, and considering the periods, we let
\[
T = \begin{pmatrix}
    e && 0 && b && 0\\
    0 && e && 0 && 1\\
    1 && 0 && 0 && 0\\
    0 && b && 0 && 0\\
\end{pmatrix}.
\]
We have $T^2 = eT + b$, and $T$ is self-adjoint.
It is actually a generator of multiplication by $\mathcal O_D$. Indeed, let us recall an argument from \cite[Theorem 8.3]{MCM07}. 
A self-adjoint endomorphism has the following form 
\[S = \begin{pmatrix}
 \epsilon & 0 & \delta & -\beta\\
 0 & \epsilon & -\gamma & \alpha\\
 \alpha & \beta & \eta & 0\\
 \gamma & \delta & 0 & \eta
\end{pmatrix}.\]
If $S$ is a generator of $\mathcal O_D$, we can replace it with $S - \eta I$ and suppose that $\eta = 0$.
Since $S^\vee(\omega) = \lambda\omega$, considering periods, we get $S = T$.
Therefore the monodromy group of $L$ commutes with $T$.
The matrix $T$ is also explicited in \cite[Section 5]{MCM05}.
\subsection{A hands-on approach}
We now indicate a more direct approach to show that the monodromy group of $L$ must commute with $T$, when $D$ is not a square.
Let $f\in \mathrm{Aff}^+(X, \omega)$, $A = d(f)$, and write
\[
A = \begin{pmatrix}
a_{11} & a_{12}\\
a_{21} & a_{22}
\end{pmatrix}.
\]
We also write $\rho(f) = (u_{ij})_{1\leqslant i,j\leqslant 4}$.
Let $\gamma = a_1$. We have
\[A\cdot \mathrm{Per}(\gamma) = A
\cdot \begin{pmatrix}
\lambda\\
0
\end{pmatrix} = \mathrm{Per}(\rho(f)\cdot \gamma) =
\begin{pmatrix}
u_{11}\lambda + u_{31}b\\
u_{21}\lambda + u_{41}
\end{pmatrix}
\]
and it follows from the relation $\lambda^2 = \lambda e + b$ that

\[\begin{pmatrix}
a_{11}\\
a_{21}
\end{pmatrix}
=
\begin{pmatrix}
u_{11} - u_{31} e+ u_{31}\lambda\\
u_{21} - u_{41}\frac{e}{b} + u_{41}\frac{\lambda}{b} 
\end{pmatrix}.
\]Similarly, setting $\gamma = b_1$, we get 
\[
A\cdot \begin{pmatrix}
0\\
\lambda
\end{pmatrix} =
\begin{pmatrix}
u_{32}b + u_{12}\lambda\\
u_{42} + u_{22}\lambda
\end{pmatrix}.
\]
Just as before, we obtain
\[\begin{pmatrix}
a_{12}\\
a_{22}
\end{pmatrix}
=\begin{pmatrix}
u_{12} - u_{32}e + u_{32}\lambda \\
u_{22} - \frac{e}{b}u_{42} + \frac{\lambda}{b}u_{42}
\end{pmatrix}.
\]
Applying the same reasoning with $\gamma = a_2$ and $\gamma = b_2$ will give us relations between the numbers $u_{ij}$. Namely, we obtain with $\gamma = a_2$
\[
\begin{pmatrix}
a_{11}\\
a_{21}
\end{pmatrix}
=
\frac{1}{b}\begin{pmatrix}
u_{33}b +  u_{13}\lambda\\
u_{43} + u_{23}\lambda
\end{pmatrix}.
\]
Finally for $\gamma = e_4$, we get 
\[
\begin{pmatrix}
a_{12}\\
a_{22}
\end{pmatrix}
=
\begin{pmatrix}
u_{34}b +  u_{14}\lambda\\
u_{44} + u_{24}\lambda
\end{pmatrix}.
\]
We now have four explicit equations relating the coefficients of $\rho(f)$.
Since $\lambda$ is not rational, these equations split and give eight linear equations in the coefficients $u_{ij}$.
We can verify directly that they are another instance of the action on homology commuting with the matrix $T$.
\begin{prop}
These equations are satisfied if and only if $\rho(f)$ and $T$ commute.
\end{prop}
This approach is more elementary in the sense that one does not need to know about real multiplication to show that the matrix $T$ must commute with $\rho(f)$.
\section{Echoes of the Weierstrass curves}\label{sectionWD}
In this section we compute the monodromy groups modulo two of the generators of Weierstrass curves, and prove \cref{thdihedral}.
Then we classify the echoes of the components of $W_D$ and prove \cref{thWD} and \cref{thSTS}.
After that, we determine which covers of reduced square-tiled surfaces in $\mathcal H(2)$ have lattice of absolute periods $\mathbb Z+i\mathbb Z$, and prove \cref{thTypes}.
Finally, we remark that one can count the sizes of the orbits of reduced square-tiled in $\widetilde {\mathcal H}(2)$ surfaces, applying a result of Leli\`evre and Royer.
\subsection{Monodromy groups}
We work again with the generators of $W_D$ given by the surfaces of \cref{Lshaped}, and the marking of \cref{explicit}.
Recall from \cref{corHyp} that there are five elements of $H_1(\Sigma, \mathbb Z/2\mathbb Z)$ that lift $L$ to the hyperelliptic component, while the other ten lift it to the odd component. 
Using the criterion given in \cref{corHyp}, we verify that in our basis $(a_1, b_1, a_2, b_2)$ of $H_1(\Sigma, \mathbb Z/2\mathbb Z)$, those five elements are
$(1, 0, 1, 0)$, $(1, 0, 1, 1)$, $(1,0, 0, 1)$, $(0, 1, 0, 0)$ and $(1, 1, 0, 0)$.
Note that it follows from \cref{sameDisc} that the monodromy group of $L$ preserves these covers.

As a consequence of commutation with real multiplication, we show that the monodromy groups modulo two of $L$ are generated by the horizontal and vertical twists when $D\not \equiv 1\mod 8$.
\begin{prop}
Let $D \geqslant 5$ be such that $D \equiv 0, 4, 5\mod 8$ and $(X, \omega)\in \Omega\mathcal M_2$ generating $W_D$.
The monodromy group modulo two of $(X, \omega)$ is dihedral, generated by two Dehn twists.
\end{prop}

\begin{proof}
Let us suppose that $D\equiv 0\mod 8$. Then $W_D$ is generated by the $L$-shaped translation surface of \cref{Lshaped}, where $e = 0$ and $4b = D$.
Since $D\equiv 0\mod 8$, we have $b\equiv 0\mod 2$. 
The elements in the monodromy group modulo two induced by the horizontal and vertical Dehn twists have been computed in \cref{explicit}. 
They generate a group of order eight in $\mathrm{Sp}(H_1(\Sigma, \mathbb Z/2\mathbb Z))$. The group $\mathrm{Sp}(H_1(\Sigma, \mathbb Z/2\mathbb Z))$ is of order 720.
The monodromy group modulo two $G$ of $L$ is included in its subgroup of elements that commute with $T$, and preserve the set of covers lifting to the hyperelliptic component. This latter group can be checked to be of order eight. Therefore $G = \langle V, H\rangle$.
We proceed similarly for $D \equiv 4, 5\mod 2$ with $(e, b)$ being respectively $(0, 1), (1, 1)$ modulo two.
\end{proof}
When $D\equiv 1\mod 8$, we have seen in \cref{explicit} that the horizontal and vertical Dehn twists do not generate the monodromy group modulo two of $L$. However, we can still characterise it.

\begin{prop}
Let $(X, \omega)\in \Omega\mathcal M_2$ generating a Weierstrass curve of discriminant $D$, with $D\equiv 1\mod 8$.
The monodromy group modulo two of $(X, \omega)$ is dihedral of order twelve.
\end{prop}

\begin{proof}
As before, we consider the $L$-shaped surface of \cref{Lshaped}, with $(e,b)$ such that $e\in \{-1, 1\}$ and $D = e^2 + 4b$.
The group of elements in $\mathrm{Sp}(H_1(\Sigma, \mathbb Z/2\mathbb Z))$ that commute with $T$ and preserve the covers lifting to the hyperelliptic component is then checked to be dihedral of order twelve. However the subgroup $\langle H, V \rangle$ is of order six, and is strictly included in the monodromy group modulo two by \cref{D1}.
Therefore the monodromy group modulo two of $L$ is of order twelve.
\end{proof}

Concretely, this monodromy group is generated for example by the matrices $H$ and 
\[\begin{pmatrix}
    0 & 1 & 0 & 1\\
    0 & 0 & 1 & 0\\
    0 & 1 & 0 & 0\\
    1 & 0 & 1 & 0
\end{pmatrix}.\]

\subsection{Classification of the echoes}
We now use the above generators of the monodromy group modulo two to classify the echoes of Weierstrass curves.
It suffices by  \cref{corEchoes} to classify the orbits of $H_1(\Sigma, \mathbb Z/2\mathbb Z)\setminus \{0\}$ under the action of these groups. In \cref{tabEchoes}, we summarise results obtained by doing so, for the different congruences of $D$ modulo eight.
In the first line we recall the order of $G$, the monodromy group modulo two of their generators.
We then show the sets of elements of $H_1(\Sigma, \mathbb Z/2\mathbb Z)$ that lift $L$ to generators of the same Teichm\"uller curve.
For this purpose, we label the elements of $H_1(\Sigma, \mathbb Z/2\mathbb Z)$ as follows.
The element $(x_1, x_2, x_3, x_4)\in H_1(\Sigma, \mathbb Z/2\mathbb Z)$ corresponds to the number $\sum 2^i x_i$.
For example, $(1, 0, 1, 1)$ is numbered $1 + 4 + 8 = 13$. 
We then write for example $\{2, 5\}$ to indicate that the Teichm\"uller curve generated by the lift of $L$ by the covers corresponding to $(0, 1, 0, 0)$ and $(1, 0, 1, 0)$ are the same, and that these are the only covers giving this curve.
\begin{table}[!h]
    \centering
\hspace*{-1.5cm}
\begin{tabular}{| c || c | c | c | c |}
\hline
     $D\mod 8$&  $0$ & $1$ & $4$ & $5$\\
\hline
     Order of $G$ & $8$ & $12$ & $8$ & $10$\\
\hline
     Echoes in $\mathcal H(2, 2)^\mathrm{hyp}$ & $\{2\}$, $\{3, 5, 9, 13\}$ & $\{2, 5\}$, $\{3, 9, 13\}$  & $\{2, 3, 9, 13\}$, $\{5\}$ & $\{2, 3, 5, 9, 13\}$\\
\hline
    Echoes in $\mathcal H(2, 2)^\mathrm{odd}$ & \makecell{$\{1, 7, 11, 15\}$, \\ $\{4, 6\}$, $\{8, 10, 12, 14\}$} & \makecell{$\{1, 6, 8, 11, 12, 15\}$,\\$\{4, 10, 14\}$, $\{7\}$}
    & \makecell{$\{1, 4, 11, 14\}$\\ $\{6, 7, 8, 12\}$, $\{10, 15\}$} & \makecell{$\{1, 8, 11, 12, 14\}$\\ $\{4, 6, 7, 10, 15\}$}\\
\hline
\end{tabular}
    \caption{The different echoes of $W_D$.}
    \label{tabEchoes}
\end{table}
\begin{center}

\end{center}

The orbits in \cref{tabEchoes} correspond to the echoes of each Weierstrass curve, which proves \cref{thWD}.
It also allows us to prove \cref{thSTS}.
\begin{proof}
Reduced square-tiled surfaces belong to the same $\mathrm{SL}_2(\mathbb Z)$-orbit if and only if they generate the same Teichm\"uller curve.
They cover a square-tiled surface in $\mathcal H(2)$ with $d$ squares if and only if the corresponding Teichm\"uller curve is an echo of a component of $W_{D}$ with $D = d^2$. The $\mathrm{SL}_2(\mathbb Z)$-orbits of reduced square-tiled surfaces covering one in $\mathcal H(2)$ are thus classified by the echoes of Weierstrass curves with discriminant $D$.
Since $D$ is a square, $D\equiv 0, 1, 4\mod 8$, and each component of $W_D$ has five echoes.
If $d$ is even or $d=3$, then $W_D$ has only one component.
When $d\geqslant 5$ is odd, then $W_D$ has two components, giving ten $\mathrm{SL}_2(\mathbb Z)$-orbits of reduced square-tiled surfaces with $2d$ tiles.
\end{proof}

When $D\not \equiv 5\mod 8$ and $(X, \omega)$ generates $W_D$, there is a single lift of $(X, \omega)$ with Veech group $\mathrm{SL}(X, \omega)$.
This creates a copy of $W_D$ in $\mathcal M_3$.
 We will see in \cref{sectionHigher} that primitive Teichm\"uller curves always have copies in higher genera moduli spaces.

\subsection{Primitivity}
Let $(X, \omega) = L(b, e)$ be the $L$-shaped surface generating a Weierstrass curve of discriminant $D$ from \cref{explicit}.
Let us denote by $\Lambda = \mathrm{Per}\left (H_1(\Sigma, \mathbb Z)\right )$ the set of its periods.
If $p\colon S\to \Sigma$ is a covering, denote by $\Lambda_p$ the set of periods of $p^*(\omega)$.
In other words, if $p$ is the cover associated with $f\in H^1(\Sigma, \mathbb Z/2\mathbb Z)$, then 
\[
\Lambda_p = \mathrm{Per}(\{\gamma\in H_1(\Sigma, \mathbb Z) \mid f(\gamma) = 0\}).
\]
Note that if $p$ is the cover associated with $\gamma$, then $f\colon \alpha\mapsto \alpha\cdot \gamma$.
When $D$ is not a square, then $\Lambda_p$ is always of index two in $\Lambda$.
\begin{lemma}
The subgroup $\Lambda_p\subset \Lambda$ is of index at most two, and is of index two when $D$ is not a square.
\end{lemma}
\begin{proof}
Since $H_1(S, \mathbb Z)$ is sent to an index two subgroup of $H_1(\Sigma, \mathbb Z)$ by $p$, the index of $\Lambda_p$ in $\Lambda$ is at most two.
If $D$ is not a square, then the period map $\mathrm{Per}\colon H_1(\Sigma, \mathbb Z)\to \Lambda$ is a group isomorphism, hence $[\Lambda\colon \Lambda_p] = 2$.
\end{proof}
We thus now assume that $D$ is a square and $D = d^2$. 
Since $\lambda^2 = e\lambda + b$, we have
\[\Lambda = \lambda\mathbb Z + b\mathbb Z + i\left (\mathbb Z + \lambda \mathbb Z\right ) = \lambda\mathbb Z + i\mathbb Z.\]
Let us say that the cover $p$ is \emph{primitive} when $\Lambda_p = \Lambda$.
In other words, a cover is primitive if the lattice of absolute periods of the cover $\tilde L = p^*(L)$ is the same as the one of $L$.
We now identify the covers $p$ that are primitive.
\begin{lemma}
The covers associated with $\gamma$ of the form $(x, 0, 0, y)$ are primitive.
\end{lemma}

\begin{proof}
Let us show in this case that $\Lambda_p = \Lambda$.
The $f\in H^1(\Sigma, \mathbb Z/2\mathbb Z)$ corresponding to $\gamma$ is
\[(x_1, y_1, x_2, y_2)\mapsto x_2 y - y_1 x.\]
Therefore, $(1, 0, 0, 0)$ is in $\ker f$. Its period is $\lambda$.
Similarly, $(0, 0, 0, 1)$ is in $\ker f$. Its period is $i$.
Thus $\Lambda_p = \lambda\mathbb Z + i\mathbb Z$.
\end{proof}
They correspond to $(1, 0, 0, 0)$, $(0, 0, 0, 1)$ and $(1, 0, 0, 1)$ which are labeled in \cref{tabEchoes} by $1$, $8$, and $9$ respectively.
Let us now consider the cover labeled $5$.
\begin{lemma}
The cover associated with $\gamma = (1, 0, 1, 0)$ is primitive if and only if $e+d\equiv 0\mod 4$.
\end{lemma}
\begin{proof}
Here $f\colon (x_1,y_1,x_2,y_2)\mapsto - y_1 - y_2$.
We have $\gamma = (1, 0, 0, 0)\in \ker f$.
Thus $\lambda\in \Lambda_p$.
We have $\Lambda_p = \Lambda$ if and only if $i\in \Lambda_p$.
This is the case when one can write $i = x_1\lambda + x_2 b + i(y_1\lambda + y_2)$ with $y_1 + y_2\in 2\mathbb Z$.
That is, when $1 = y_1\lambda + y_2$, with $y_1+y_2\in 2\mathbb Z$.
When $\lambda$ is odd, reducing modulo two gives $y_1 + y_2 \equiv  1\mod 2$, and there is no solution.
When $\lambda$ is even, one can take $y_1 = 1$ and $y_2 = 1-\lambda$.
The result follows since $\lambda = \frac{e+d}{2}$ is even exactly when $e+d$ is a multiple of four.
\end{proof}
We now turn to the case of the cover labeled two.
\begin{lemma}
The cover associated with $\gamma = (0, 1, 0, 0)$ is primitive if and only if $d - e\equiv 2\mod 4$.
\end{lemma}
\begin{proof}
This cover is primitive if and only if one can write $x_1\lambda + x_2 b = \lambda$ with $x_1\in 2\mathbb Z$, that is, $x_1 + x_2(\lambda-e) = 1$.
There is no solution if $\lambda - e$ is even, and $(x_1,x_2) = (1 + \lambda - e, 1)$ is a solution otherwise. Therefore the cover is primitive if and only if $\lambda-e = \frac{d-e}{2}$ is odd, that is when $d-e\equiv 2\mod 4$.
\end{proof}
Let us now turn to the cover number five.
\begin{lemma}
The cover associated with $\gamma = (1, 0, 1, 0)$ is primitive if and only if $d + e\equiv 0\mod 4$.
\end{lemma}
\begin{proof}
It is primitive if and only if $y_1\lambda + y_2 = 1$ has a solution with $y_1+y_2\in 2\mathbb Z$.
This happens exactly when $\lambda$ is even.
\end{proof}
We can now tell which of the covers lifting to the hyperelliptic component are primitive.
Indeed primitive covers are preserved by the action of $\mathrm{Aff}^+(X, \omega)$.
Moreover the orbits of covers lifting to the hyperelliptic component are represented in every case by the covers numbered $2, 5$ and $9$, see \cref{tabEchoes}.
We now turn to the covers lifting to the odd component, starting with the one numbered four.
\begin{lemma}
The cover associated with $\gamma = (0, 0, 1, 0)$ is primitive if and only if $d+e\equiv 2\mod 4$.
\end{lemma}
\begin{proof}
It is primitive if and only if $1 = y_1 \lambda + y_2$ with $y_2\in 2\mathbb Z$.
This is possible exactly when $\lambda$ is odd, that is when $d+e\equiv 2\mod 4$.
\end{proof}

We now consider the cover labeled seven.
\begin{lemma}
The cover associated with $\gamma = (1, 1, 1, 0)$ is primitive if and only if $e=0$ or $e + d\equiv 0\mod 4$.
\end{lemma}
\begin{proof}
We have $f\colon (x_1, y_1, x_2, y_2)\mapsto y_1 - x_1 + y_2$.
The cover is primitive when the following two equations have a solution with $(x_1, y_1, x_2, y_2)$ in $\ker f$
\begin{align*}
    i &= x_1\lambda + x_2 b + i(y_1\lambda + y_2),\\
     \lambda &= x_1\lambda + x_2 b + i(y_1\lambda + y_2).
\end{align*}
There is a solution to the first equation when $\lambda$ is even: 
\[(x_1, y_1, x_2, y_2) = (0,1, 0, 1-\lambda).\]
There is also a solution when $\lambda$ is odd and $e = 0$:
\[(x_1, y_1, x_2, y_2) =(1, \lambda, -\lambda, 1-\lambda).\]
However if $\lambda$ is odd and $e=\pm 1$, then from $x_1\lambda + x_2 b= 0$ we get  
$x_1 + x_2(\lambda - e) = 0$ and $x_1$ must be even.
One must have $y_1\lambda + y_2 = 1$ and $y_1+y_2\equiv 0\mod 2$, which has no solution. 
For the second equation, there is a solution when $\lambda-e$ is odd:
\[(x_1, y_1, x_2, y_2) = (1+\lambda - e, 0,  -1, 0).\]
When $e=0$ and $\lambda$ is even, we also have the solution
\[(x_1, y_1, x_2, y_2) = (1+\lambda, 1, -1, -\lambda).\]
If $e=\pm 1$ and $\lambda$ is odd, a solution must have $x_1 \equiv 1\mod 2$ and $y_1\lambda + y_2 = 0$ with $y_1+y_2\equiv 1\mod 2$. This is not possible since $\lambda$ is odd.
Summarising, we have a solution to both equations if and only if $e=0$, or when $\lambda$ is even.
\end{proof}
Finally, we analyse the primitivity of the cover number ten.
\begin{lemma}
The cover associated with $\gamma = (0, 1, 0, 1)$ is primitive if and only if $d-e\equiv 0\mod 4$.
\end{lemma}
\begin{proof}
It is primitive when one can write $x_1\lambda + x_2 b= \lambda$ with $x_1 + x_2\in 2\mathbb Z$. This happens when $1 = x_1 + x_2 (\lambda - e)$, which has a solution if and only if $\lambda-e$ is even.
\end{proof}
We can now tell which covers lifting to the odd component is primitive. Indeed the orbits of these covers under the action of $\mathrm{Aff}^+(X, \omega)$ are represented by the covers $1, 4, 8, 7, 10$, see \cref{tabEchoes}.
We summarise in \cref{tabEchoesPrimitifs} the echoes of $L$ obtained by a primitive cover.
Observe that if $e=\pm 1$, then $d+e\equiv d - e + 2\mod 4$.
\begin{table}[!h]
    \centering
\hspace*{-1cm}
\begin{tabular}{| c || c | c | c | c |}
\hline
     $(d, d-e)\mod 4$&  $(0, 0)$ & $(\pm 1, 0)$ & $(2, 2)$ & $(\pm 1, 2)$\\
\hline
     Echoes in $\mathcal H(2, 2)^\mathrm{hyp}$ & $\{3, 5, 9, 13\}$ & $\{3, 9, 13\}$  & $\{2, 3, 9, 13\}$ & $\{2, 5\}$, $\{3, 9, 13\}$ \\
\hline
    Echoes in $\mathcal H(2, 2)^\mathrm{odd}$ & \makecell{$\{1, 7, 11, 15\}$,\\ $\{8, 10, 12, 14\}$} & \makecell{$\{1, 6, 8, 11, 12, 15\}$,\\$\{4, 10, 14\}$}
    & \makecell{$\{1, 4, 11, 14\}$\\ $\{6, 7, 8, 12\}$} & \makecell{$\{1, 6, 8, 11, 12, 15\}$, \{7\}}\\
\hline
\end{tabular}
    \caption{The echoes of square-tiled surfaces by primitive covers.}
    \label{tabEchoesPrimitifs}
\end{table}

\cref{thTypes} follows from this analysis.
\begin{proof}
The types of branched covers of degree $d$ of \cref{thTypes} correspond to the $\mathrm{SL}_2(\mathbb Z)$-orbits of square-tiled surfaces that are lifts of a reduced square-tiled surface in $\mathcal H(2)$ by a primitive cover. 
\end{proof}

It follows from \cite[Theorem 5.3]{MCM05} that when $d$ is odd, we have $d-e\equiv 0\mod 4$ if and only if the spin invariant of the curve generated by $L = L(\frac{d^2-1}{4}, e)$ vanishes. This means that reduced square-tiled surfaces in the $\mathrm{GL}_2^+(\mathbb R)$-orbit of $L$ have exactly one Weierstrass point on $\mathbb Z+i\mathbb Z$, see \cite[Theorem 6.1]{MCM05}, while the ones in the other component of $W_D$ have three integral Weierstrass points.
This way of distinguishing the orbits of square-tiled surfaces was introduced by Hubert and Lelièvre in \cite{HL06}.

\subsection{Counting square-tiled surfaces}
Remark that our proof gives the indices of the Veech groups of the covers.
In particular, since Leli\`evre and Royer in \cite{LR06} counted the sizes of the $\mathrm{SL}_2(\mathbb Z)$-orbits of reduced square-tiled surface in $\mathcal H(2)$, we have formulas for the sizes of $\mathrm{SL}_2(\mathbb Z)$-orbits of square-tiled surface in $\widetilde{\mathcal H}(2)$.
As an example, let us count the number of square-tiled surfaces in $\widetilde{\mathcal H}(2)$, with $d=22$ squares, odd Arf invariant and lattice of absolute periods $\mathbb Z+i\mathbb Z$.
By the results gathered in \cref{tabEchoesPrimitifs}, they form four $\mathrm{SL}_2(\mathbb Z)$-orbits.
From \cite{LR06}, the two $\mathrm{SL}_2(\mathbb Z)$-orbits of reduced square-tiled surfaces in $\mathcal H(2)$ with an odd number $n$ of squares are of sizes
$a_n = \frac{3}{16}(n-1)n^2\prod_{p\mid n} (1-\frac{1}{p^2})$ and $b_n = \frac{n-3}{n-1}a_n$. 
The numbers $a_n$ are the sizes of the orbits with vanishing spin invariant, while $b_n$ are the sizes of the others. 
Thus the $a_n$, resp. $b_n$, correspond to the orbits whose lifts are classified on the second, resp. last, column of \cref{tabEchoes}.
For $n=11$, we have $a_n = 225$ and $b_n = 180$.
It follows from \cref{veech_cover} that the sizes of our orbits lifting origamis with vanishing spin invariant are  $1350 = 6\cdot 225$ and $900 = 4\cdot 225$. The remaining orbits are of sizes $1080 = 6\cdot 180$ and $180$.
\section{Remarks in higher genus}\label{sectionHigher}

In this section we make some remarks on Teichm\"uller curves in higher genus, obtained by covering constructions.
We first show that there are arbitrarily large numbers of echoes of Teichm\"uller curves of genus two, in higher genus.
We then show that, as a corollary of a theorem of Funar and Lochak, one can find copies of a given primitive Teichm\"uller curve in moduli spaces of higher genera.

\subsection{There are arbitrarily large numbers of pure echoes}
As a consequence of \cref{commutationReal}, we show that a Teichm\"uller curve $V\subset \mathcal M_g$ can have an arbitrarily large number of \emph{pure} echoes in higher genera moduli spaces. 
Since square-tiled surfaces are dense in each moduli space $\Omega\mathcal M_g$, there can be infinite numbers of echoes of a given Teichm\"uller curve in a given $\mathcal M_g$.
However, this is not possible with pure echoes.
\begin{lemma}Let $V\subset \mathcal M_g$ be a Teichm\"uller curve.
There is only a finite number of pure echoes of $V$ in $\mathcal M_{g'}$ for each $g'$.
\end{lemma}
\begin{proof}
There is only one Teichm\"uller curve in $\mathcal M_1$, thus we suppose that $g\geqslant 2$.
There are only finitely many covers of genus $g'$ of a given genus $g$ surface.
\end{proof}

Recall that we call cyclic echoes of a Teichm\"uller curve $V\subset \mathcal M_g$ the Teichm\"uller curves generated by $p^*(\omega)$ where $\omega$ generates $V$ and $p$ is a normal cover with cyclic deck transformation group.
The number of cyclic echoes of Teichm\"uller curves generated by a square-tiled surface in $\mathcal H(2)$ can be made arbitrarily large.

\begin{prop}\label{PropArbitrarilyLargeNumber}
Let $b > 1$ be a square and $D = 4b$. For any odd $n\geqslant 3$ coprime with $b$,
the number of cyclic echoes of $W_D$ in $\mathcal M_{n+1}$ is at least $\varphi(n)$.
\end{prop}
In \cref{PropArbitrarilyLargeNumber}, the function $\varphi$ is the Euler's totient function.
Let us first observe that as in \cref{corEchoes}, the echoes are classified by the action of $\mathrm{Aff}^+(X, \omega)$ on the set of primitive elements of $ H_1(\Sigma, \mathbb Z/n\mathbb Z)$.
Actually, the proof of \cref{lemmaNoAuto} shows the following more general result.
\begin{cor}\label{transG}
Let $(X, \omega)\in \Omega\mathcal T(\Sigma)$ with a unique double zero, and $p\colon S_g\to \Sigma$ be an unramified normal cover with deck transformation group $G$.
Suppose that $(X, \omega)$ generates a Teichm\"uller curve $V$ that is not $W_9$.
\begin{enumerate}
    \item The translation group of $p^*(\omega)$ is isomorphic to $G$.
    \item The echoes of $V$ by such covers are in correspondence with the action of $\mathrm{Aff}^+(X, \omega)$ on the set of these covers.

\end{enumerate}
\end{cor}
\begin{proof}
The group of deck transformations $G$ is naturally embedded in the translation group.
On the translation surface $p^*(\omega)$, there are $|G|$ saddle connections $\sigma_1,\ldots, \sigma_{|G|}$ lifting the one of the proof of \cref{lemmaNoAuto}.
Any translation must permute the $\sigma_i$, and is determined by the image of $\sigma_1$. The translation group has thus order at most $|G|$.
For the second affirmation, it suffices to notice that any isomorphism $F\colon S_g\to S_g$ sending one lift to the other must then conjugate the corresponding deck transformations as in \cref{lemmeMartens}.
\end{proof}
We can now prove \cref{PropArbitrarilyLargeNumber}.
\begin{proof}
Let $b' = \sqrt b$.
The matrix $T$ admits $u_1 = (b', 0, 1, 0)$ and $u_2 = (0, b', 0, 1)$ as eigenvectors associated with the eigenvalue $b'$. 
It also admits $u_3 = (b', 0, -1, 0)$ and $(0, b', 0, -1)$ as eigenvectors with eigenvalue $-b'$.
Since the monodromy group commutes with $T$, any matrix $A\in \rho(\mathrm{Aff}^+(X, \omega))$ must preserve these sets.

The following matrix is invertible modulo $n$ since its determinant is $4b$:
\[
\begin{pmatrix}
    b' & 0 & b' & 0\\
    0 & b' & 0 & b'\\
    1 & 0 & -1 & 0 \\
    0 & 1 & 0 & -1
\end{pmatrix}.
\]
Thus $(u_1, u_2, u_3, u_4)$ forms a basis of $H_1(\Sigma, \mathbb Z/n\mathbb Z)$, and we now write $A$ in this basis.
The matrix $A$ is now block-diagonal. Thus the elements $(1, 0, x, 0)\in H_1(\Sigma, \mathbb Z/n\mathbb Z)$ with $x\in \mathbb Z/n\mathbb Z$ are not in the same orbit for $x$ generating different subgroups of $\mathbb Z/n\mathbb Z$. There are $\varphi(n)+1$ of these groups.
\end{proof}
Note that one could improve the bound to $2\varphi(n) + 1$ in the preceding argument by also considering the elements of the form $(x, 0, 1, 0)$. 
One can then count primitive elements of the form $(x, 0, y, 0)$, where both $x$ and $y$ do not generate $\mathbb Z /n\mathbb Z$.
We can also use the fact that $A$ is symplectic to restrict the possibilities of its action on homology even further.

\subsection{Any Teichm\"uller curve appears in infinitely many moduli spaces}
Given $V$ a hyperbolic surface of finite type, it is known that there are only finitely many Teichm\"uller curves $V\to \mathcal M_g$ for each genus $g$, see \cite{MC09}.
A result of Funar and Lochak \cite[Theorem 1.4]{FL18} states that every closed surface $\Sigma_g$ of genus $g\geqslant 2$ admits infinitely many surjections $\rho\colon \pi_1(\Sigma_g)\to G$ with $G$ simple and finite, such that $\ker \rho$ is a characteristic subgroup of $\pi_1(\Sigma_g)$. 
As a consequence of this result and our observations, we show that for a Teichm\"uller curve $V\to \mathcal M_g$, there exist infinitely many $g' > g$ and Teichm\"uller curves $V\to \mathcal M_{g'}$.
It is proven in \cite{H06} that there are infinitely many copies of the Teichm\"uller curve in $\mathcal M_1$ generated by any torus, namely $\mathbb H^2/\mathrm{SL}_2(\mathbb Z)$.

\begin{prop}
There are infinitely many copies of a given primitive Teichm\"uller curve in higher genera moduli spaces.
\end{prop}

\begin{proof}
This is proven by Herrlich in \cite{H06} for the only Teichm\"uller curve in $\mathcal M_1$, hence we fix $\Sigma_g$ a closed surface of genus $g\geqslant 2$.
Let $V\subset \mathcal M_g$ be a primitive Teichm\"uller curve and $\rho\colon \pi_1(\Sigma_g)\to G$ a surjection onto a finite group with characteristic kernel.
Fix $(X, \omega)\in \Omega\mathcal T(\Sigma_g)$ generating $V$.
Let $p\colon S'\to \Sigma_g$ be the cover associated with $\ker \rho$.
The Teichm\"uller curve generated by $p^*(\omega)$ is isomorphic to $V$.
Indeed, since $\ker \rho$ is characteristic, $p$ is stabilised by the whole group $\Mod(\Sigma_g)$.
Therefore by \cref{rmkInclusion}, the Veech group of $p^*(\omega)$ contains $\mathrm{SL}(X, \omega)$.
By results of M\"oller \cite[Theorem 2.6]{M06}, see also \cite[Corollary 2.2]{MCMPrym}, the Veech group of $p^*(\omega)$ is a subgroup of $\mathrm{SL}(X, \omega)$.
We thus get a copy of $V = \mathbb H^2/\mathrm{SL}(X, \omega)$ in $\mathcal M_{g'}$, where $g' = (g-1)|G| + 1$.
There exist arbitrarily large such $G$ by \cite{FL18}.
\end{proof}
\bibliographystyle{alpha}
\bibliography{biblio}
\end{document}